\documentclass[English]{amsart}

\usepackage{amsmath,amssymb,verbatim,amsthm,mathrsfs,mathtools}
\usepackage[bookmarks,hidelinks]{hyperref}
\usepackage[capitalize]{cleveref}

\newcommand{\Rb}{\mathbb{R}}
\newcommand{\Ob}{\mathbb{O}}

\newcommand{\M}{\mathbb{M}}

\newcommand{\abar}{\bar{a}}
\newcommand{\bbar}{\bar{b}}
\newcommand{\cbar}{\bar{c}}
\newcommand{\ebar}{\bar{e}}
\newcommand{\fbar}{\bar{f}}
\newcommand{\gbar}{\bar{g}}
\renewcommand{\hbar}{\bar{h}}

\newcommand{\xbar}{\bar{x}}
\newcommand{\ybar}{\bar{y}}
\newcommand{\zbar}{\bar{z}}

\DeclareMathOperator{\res}{\upharpoonright}

\DeclareMathOperator{\Th}{Th}

\newcommand{\cx}{\heartsuit}

\newtheorem{thm}{Theorem}[section]
\newtheorem{prop}[thm]{Proposition}
\newtheorem{lem}[thm]{Lemma}
\newtheorem{cor}[thm]{Corollary}
\newtheorem{fact}[thm]{Fact}

\newtheorem{quest}[thm]{Question}
\theoremstyle{definition}
\newtheorem{defn}[thm]{Definition}
\newtheorem{nota}[thm]{Notation}
\theoremstyle{remark}

\newcommand{\e}{\varepsilon}

\DeclareMathOperator{\tp}{tp}

\newcommand{\oo}{$1\text{\upshape-}1$}

\newcommand{\RFR}{\mathbb{R}\mathsf{FR}}

\newcommand{\Lc}{\mathcal{L}}

\newcommand{\Uc}{\mathcal{U}}

\newcommand{\Fraisse}{Fra\"\i ss\'e}

\newcommand{\mxmn}[3]{[#1]_{#2}^{#3}}
\newcommand{\mxmnu}[1]{[#1]_{0}^{1}}

\DeclareMathOperator{\dcl}{dcl}
\DeclareMathOperator{\acl}{acl} %

\DeclareMathOperator{\Aut}{Aut}

\def\Ind{\setbox0=\hbox{$x$}\kern\wd0\hbox to 0pt{\hss$\mid$\hss}
  \lower.9\ht0\hbox to 0pt{\hss$\smile$\hss}\kern\wd0}

\def\Notind{\setbox0=\hbox{$x$}\kern\wd0\hbox to 0pt{\mathchardef
    \nn=12854\hss$\nn$\kern1.4\wd0\hss}\hbox to
  0pt{\hss$\mid$\hss}\lower.9\ht0 \hbox to 0pt{\hss$\smile$\hss}\kern\wd0}

\def\ind{\mathop{\mathpalette\Ind{}}}

\newcommand{\indast}{{\textstyle\ind^{\!\!\ast}}}

\newcommand{\fpc}{\mathrm{4PC}}

\newcommand{\RF}{\Rb\mathsf{F}}
\newcommand{\RFs}{\Rb\mathsf{F}^\ast}

\DeclareMathOperator{\Mod}{Mod}

\newcommand{\interp}[3]{#1{:}#2{:}#3}%

\DeclareMathOperator{\dis}{dis}

\DeclareMathOperator{\corr}{cor}

\newcommand{\rKRFR}{\rho_K}%

\newcommand{\dc}{\#^{\mathrm{dc}}}

\DeclareMathOperator{\cf}{cf}

\newcommand{\To}{\Rightarrow}

\begin{document}

\title{A simple continuous theory}
\address{Department of Mathematics\\
  University of Maryland\\
  College Park, MD 20742, USA}
\author{James Hanson}
\email{jhanson9@umd.edu}
\date{\today}


\begin{abstract}
  In the context of continuous first-order logic, special attention is often given to theories that are somehow continuous in an `essential' way. A common feature of such theories is that they do not interpret any infinite discrete structures. We investigate a stronger condition that is easier to establish and use it to give an example of a strictly simple continuous theory that does not interpret any infinite discrete structures: the theory of richly branching $\Rb$-forests with generic binary predicates. We also give an example of a superstable theory that fails to satisfy this stronger condition but nevertheless does not interpret any infinite discrete structures.
\end{abstract}

\maketitle

\section*{Introduction}
\label{sec:intro}

In continuous first-order model theory, there is a common desire to find examples of various phenomena occurring in `essentially continuous' theories, where `essential continuity' is the informally defined property of resembling a discrete theory as little as possible. Most of the standard motivating examples of metric structures---infinite-dimensional Hilbert spaces, atomless probability algebras, operator algebras, Banach lattices, the Urysohn sphere, and $\Rb$-trees---are regarded as essentially continuous \cite{Ealy2012}. Despite not settling on a precise definition of essential continuity, continuous logicians have been asking the following question for a long time.
\begin{quest}\label{quest:ess-cont-simple}
  Is there an essentially continuous strictly simple theory?
\end{quest}
As discussed in the introduction of \cite{Ealy2012}, natural attempts to produce an example fail: The randomization of an IP theory is TP$_2$ \cite[Cor.~4.13]{Yaacov2013}, implying that no randomization is strictly simple, and the theories produced by adding generic automorphisms to the standard examples of essentially continuous stable theories are themselves stable (whereas in discrete logic this often produces a strictly simple theory).  

Obviously we would like to address \cref{quest:ess-cont-simple} in this paper. The issue of course is that \ref{quest:ess-cont-simple} is not fully specified. In all honesty, it is unclear that there will ever be a definitively compelling notion of essential continuity, and as such we will not attempt to provide one here. %

In earlier work of the author studying uncountable (or inseparable) categoricity in continuous logic, the condition of not interpreting the theory of any infinite discrete structure (which by an abuse of terminology we will refer to as not `interpreting an infinite discrete structure') was used as a precisely defined proxy for essential continuity \cite{CatCon}. In that context, this was a rather natural choice, as any $\omega$-stable theory that interprets an infinite discrete structure has a strongly minimal imaginary sort,\footnote{Although note that the converse does not hold.} allowing for a Baldwin-Lachlan characterization of uncountable categoricity for those theories in particular. Outside of the context of $\omega$-stable theories (or, more generally, outside of the context of dictionaric theories), this is perhaps too strong. The difficulty is that interpretation of an infinite discrete structure relies on the ability to isolate the relevant discrete behavior in a definable set and, as is well known at this point, definable sets are generally quite poorly behaved in continuous logic. As we will show in \cref{sec:separating}, this can in fact be a fatal obstruction to interpreting an infinite discrete structure in a theory that is seemingly rather discrete in nature. Given this issue, in this paper we study a stronger precisely defined proxy for essential continuity, which is roughly speaking equivalent to having no uniformly discrete `hyperimaginary sorts.' This notion seems to be more robust; in particular, it is always witnessed by $1$-types (\cref{thm:crisp-imaginary-char-1-types}). Furthermore, as evidenced by the relative complexity of the proofs of \cref{prop:heart-no-interpret} and \cref{thm:no-crispy}, it seems to be generally easier to work with. Finally, we show that, under the assumption of dictionaricity (and so a fortiori $\omega$-stability), it is equivalent to the condition of interpreting an infinite discrete structure (\cref{prop:dict-then-interp}). %

In \cref{sec:gen-bin-R-tree}, we present an example of a strictly simple theory that exhibits this stronger non-discreteness property (and therefore also does not interpret any infinite discrete structure). The idea is itself fairly simple---add a generically chosen binary predicate to a generic `$\Rb$-forest.' Aside from verifying that the example works, the only real issue is choosing the modulus of uniform continuity of the predicate correctly so that the resulting class of structures has the amalgamation property. It turns out that $1$-Lipschitzness with regards to the $\ell^1$-metric on pairs (i.e., $d(xy,zw) = d(x,z)+d(y,w)$) is precisely what is necessary to make this work.

Finally, in \cref{sec:quests}, we describe a stronger non-discreteness property that separates Hilbert spaces, atomless probability algebras, and the Urysohn sphere from $\Rb$-trees and ask whether it is incompatible with strict simplicity.

\subsection*{Notation}
\label{sec:nota}

\begin{sloppypar}
  In continuous logic, often the minimum and maximum function are used to clamp a given value to a certain interval, such as in the expression $\max(0,\min(1,x))$.  In the author's opinion representing such a frequently used operation with 12 characters is a bit ridiculous, so we will employ the following notation.
\end{sloppypar}
\begin{nota}
  For any $r \leq s$, we write $\mxmn{x}{r}{s}$ for the function $\max(r,\min(s,x))$.
\end{nota}

\section{Crisp imaginary types}%
\label{sec:Twnacit}

\begin{defn}\label{defn:crispy}
  In a theory $T$, a type $p(x)$ is \emph{crisp} if the metric is $\{0,1\}$-valued on the set of realizations of $p(x)$ in the monster model.
\end{defn}

While it might be natural to use the word `discrete' for the previous notion, there is an issue with this, which is that discreteness in metric spaces is a topological condition and is strictly weaker than uniform discreteness (even in arbitrarily saturated metric structures). `Uniform discreteness,' while accurate, is also long and still fails to specify the scale of uniform discreteness. (In particular, the class of infinite structures with uniformly discrete metrics is not elementary.) As such, we have opted to use the adjective `crisp,' which is commonly used in the real-valued logic literature to mean $\{0,1\}$-valuedness.

\begin{defn}
  An \emph{imaginary type} is a type in an imaginary sort.
\end{defn}

\subsection{Characterization}
\label{sec:character}

\begin{prop}\label{prop:crisp-type-char-1}
  A theory $T$ has a non-algebraic crisp imaginary type if and only if there is a type $p(x)$ in some real (possibly infinitary) product sort and a formula $E(x,y)$ such that on the set of realizations of $p(x)$, $E(x,y)$ is $\{0,1\}$-valued and defines an equivalence relation with infinitely many classes.
\end{prop}
\begin{proof}
  Let $q(z)$ be a non-algebraic crisp type in some imaginary sort $I$. By the infinitary pigeonhole principle, we may assume that $q(z)$ is a complete type over some set of parameters $A$. By \cite[Prop.\ 3.2.13]{HansonThesis}, there is a definable pseudo-metric $\rho(x,y)$ on $H^n$ (for some $n \leq \omega$ where $H$ is the home sort) and a definable subset $D \subseteq H^n/\rho$ such that $I$ and $D$ are isomorphic as sorts (with the metric on $I$ corresponding to the metric $\rho$ on $D$). Therefore we may assume that $q(z)$ is a type in $D$. Let $p(x)$ be the type of elements of $H^n$ that project to realizations of $q(z)$. We clearly now have that $\rho(x,y)$ is $\{0,1\}$-valued on the realizations of $p(x)$ and defines an equivalence relation with infinitely many classes.

  For the other direction, we have that $d(x,y) = 0 \vee (p(x) \wedge p(y) \wedge E(x,y)=0)$ is a type-definable equivalence relation. By \cite[Thm.~2.20]{ben-yaacov_2005},\footnote{See also \cite[Prop.~4.8]{conant2021separation} for an easier proof.} there is a definable pseudo-metric $\rho(x,y)$ with the property that $\rho(x,y) = 0$ if and only if $d(x,y) = 0 \vee (p(x) \wedge p(y) \wedge E(x,y)=0)$. Note that for realizations of $p(x)$, we have by compactness that there is some $\e > 0$ such that if $E(x,y) = 1$, then $\rho(x,y) \geq \e$. By replacing $\rho(x,y)$ with $\min(\frac{1}{\e}\rho(x,y), 1)$, we may assume that $\rho(x,y)$ is $\{0,1\}$-valued on realizations of $p(x)$. In particular, we have that $\rho(x,y) = E(x,y)$ for realizations of $p(x)$. Let $f: H^\omega \to H^\omega/\rho$ be the definable quotient map. Let $q(y)$ be the type of $f(x)$ for any $x$ realizing $p(x)$. By construction we have that $q(y)$ is a non-algebraic crisp imaginary type.
\end{proof}

\begin{lem}
  If $T$ has a non-algebraic crisp imaginary type, then there is a type $p(x)$ in a finitary product sort and a formula $E(x,y)$ satisfying the conclusion of \cref{prop:crisp-type-char-1}.
\end{lem}
\begin{proof}
  Assume that we have $p(x)$ and $E(x,y)$ satisfying the conclusion of \cref{prop:crisp-type-char-1} with $x = x_0x_1x_2\dots$ an $\omega$-tuple of variables. Since $E(x,y)$ is $\{0,1\}$-valued on $p(x)$, there is a restricted formula $E_0(x,y)$ that is equal to $E(x,y)$ for realizations of $p(x)$. This restricted formula can only involve finitely many variables in the tuple $x$, so we may write the formula $E_0$ as $E_0(z,w)$, where $z$ and $w$ are finite initial segments of $x$ and $y$, respectively, with $|z| = |w|$. Finally, if we take $q(z)$ to be the restriction of $p(x)$ to the tuple $z$, then we have that for realizations of $q(x)$, $E_0(z,w)$ is a $\{0,1\}$-valued equivalence relation with infinitely many classes.
\end{proof}

\begin{thm}\label{thm:crisp-imaginary-char-1-types}
  Fix a theory $T$. The following are equivalent.
  \begin{enumerate}
  \item $T$ has a non-algebraic crisp imaginary type.
  \item There is a countable set of parameters $A$, a $1$-type $p(x) \in S_x(A)$, and a formula $E(x,y)$ over $A$ such that $E$ defines a $\{0,1\}$-valued equivalence relation on the realizations of $p(x)$ with infinitely many equivalence classes.
  \item There is a model $M\models T$ with $|M|\leq |T|$, a $1$-type $p(x) \in S_x(M)$, and a formula $E(x,y)$ over $M$ such that $E$ defines a $\{0,1\}$-valued equivalence relation on the realizations of $p(x)$ with more than one equivalence class.
  \end{enumerate}
\end{thm}
\begin{proof}
  2 obviously implies 1. To see that 1 implies 2, assume that we have a complete non-algebraic crisp imaginary type $p_0(\xbar)$ witnessed by the formula $E_0(\xbar,\ybar)$ over some countable set of parameters $A_0$ where $\xbar$ is an $n_0$-tuple of variables, with $n_0$ finite. We will inductively reduce the number of variables until we only have one.

  At stage $i$, given $p_i$, $E_i$, $A_i$, and $n_i$, if $n_i = 1$, then we are done. Otherwise, write $p_i$ as $p_i(\xbar,z)$ and $E_i$ as $E_i(\xbar,z,\ybar,w)$. Let $q(z)$ be the restriction of $p_i$ to the variable $y$. Let $c$ be some realization of $q$. At this point there are two cases.

  If the number of $E_i$-classes of elements of the form $\bbar c$ is infinite, then let $p_{i+1}(\xbar) = p_i(\xbar,c) \in S_{\xbar}(A_ic)$, let $E_{i+1}(\xbar,\ybar)=E_i(\xbar,c,\ybar,c)$, let $A_{i+1} = A_i c$, and let $n_{i+1}=n_i - 1$.

  If the number of $E_i$-classes of elements of the form $\bbar c$ is $k$ for some finite $k$, then the same will be true of any $c' \models q$ (since $q$ is a complete type). Let $F(z,w)$ be the relation on realizations of $q(z)$ defined by the following: $F(c,c')$ holds if and only if
  \begin{itemize}
  \item for every tuple $\bbar$ such that $\bbar c \models p_i$, there is a tuple $\bbar'$ such that $\bbar'c' \models p_i$ and $E_i(\bbar c,\bbar'c')$ and
  \item for every tuple $\bbar'$ such that $\bbar' c' \models p_i$, there is a tuple $\bbar$ such that $\bbar c \models p_i$ and $E_i(\bbar c, \bbar'c')$,
  \end{itemize}
  where quantification is taken over the monster model. This is clearly an $A_i$-invariant equivalence relation on realizations of $q(z)$. We need to show that it is actually type-definable and co-type-definable on realizations of $q(z)$. To see that it is type-definable, it is sufficient to show that if $(r_j(z,w))_{j \in J}$ is a net of types in $S_{zw}(A_i)$ converging to $r(z,w)$ such that $r_j(z,w) \vdash q(z)\wedge q(w) \wedge F(z,w)$ for each $j \in J$, then $r(z,w) \vdash F(z,w)$ as well. For each $j \in J$, let $M_j$ be an $|\Lc|^+$-saturated model containing $c_j$ and $c'_j$ such that $c_jc'_j \models r_j$. For each $j \in J$, since $M_j$ is $|\Lc|^+$-saturated, we know that there are $\bbar_{j,0},\dots,\bbar_{j,k-1}$ such that $\bbar_{j,\ell} c_j\models p_i$ for each $\ell < k$ and such that the $E_i$-classes of $\bbar_{j,0}c_j,\dots,\bbar_{j,k-1}c_j$ are distinct and represent all $E_i$-classes of tuples containing $c_j$. We also know that there are similar tuples $\bbar'_{j,0},\dots,\bbar'_{j,k-1}$ for $c'_j$. Since $F(c_j,c'_j)$ holds, we may assume that $E_i(\bbar_{j,\ell}c_j,\bbar'_{j,\ell}c'_j)$ holds for each $\ell < k$.

  Let $\Uc$ be an ultrafilter on $J$ extending the filter generated by $\{\{j' \geq j : k \in J\}: j \in J\}$. (This is a filter because $J$ is a directed set.) Let $(M_\Uc,c_\Uc,\bbar_{\Uc,0},\dots,\allowbreak\bbar_{\Uc,k-1},\allowbreak c'_\Uc,\bbar'_{\Uc,0},\dots,\bbar'_{\Uc,k-1})$ be the ultraproduct
  \[
    \prod_{j \in J}(M_j,c_j,\bbar_{j,0},\dots,\bbar_{j,k-1},c'_j,\bbar'_{j,0},\dots,\bbar'_{j,k-1})/\Uc.
  \]
  Since $E_i$ is type-definable, we have that $E(\bbar_{\Uc,\ell}c_{\Uc},\bbar'_{\Uc,\ell}c'_{\Uc})$ holds for each $\ell < k$. Furthermore, by the choice of $k$, we know that there are no more $E_i$-equivalence classes of tuples containing $c_\Uc$ and of tuples containing $c'_\Uc$, so $F(c_\Uc,c'_\Uc)$ holds. Therefore we have that $F(z,w)$ is type-definable.

  \begin{sloppypar}
    Now to see that it is co-type-definable, we need to show that if $(r_j(z,w))_{j \in J}$ is a net of types in $S_{zw}(A_i)$ converging to $r(z,w)$ such that $r_j(z,w) \vdash q(z)\wedge q(w) \wedge \neg F(z,w)$ for each $j \in J$, then $r(z,w) \vdash \neg F(z,w)$. To this end, find a family of models and tuples $(M_j,c_j,\bbar_{j,0},\dots,\bbar_{j,k-1},c'_j,\bbar'_{j,0},\dots,\bbar'_{j,k-1})_{j \in J}$ as before, with $\bbar_{j,0}c_j,\dots,\bbar_{j,k-1}c_j$ representing all $E_i$-classes of tuples containing $c_j$ and $\bbar'_{j,0}c'_j,\dots,\bbar'_{j,k-1}c'_j$ representing all $E_i$-classes of tuples containing $c'_j$. Let $\Uc$ be an ultrafilter that is compatible with the directed set $J$ as before. Since $\neg F(c_j,c'_j)$ for each $j \in J$, we have that for each $j \in J$, either there is an $\ell < k$ such that $\neg E_i(\bbar_{j,\ell}c_j,\bbar'_{j,m}c'_j)$ for all $m < k$, or there is an $m < k$ such that $\neg E_i(\bbar_{j,\ell}c_j,\bbar'_{j,m}c'_j)$ for all $\ell < k$. Since there are finitely many possibilities, there is an $X \in \Uc$ such that this behaves uniformly on $X$. Without loss of generality, we may assume that for all $j \in X$, $\neg E_i(\bbar_{j,0}c_j,\bbar'_{j,m}c'_j)$ for all $m<k$. Since this is a type-definable fact, we have that in the ultraproduct $(M_\Uc,c_\Uc,\bbar_{\Uc,0},\dots,\bbar_{\Uc,k-1},c'_\Uc,\bbar'_{\Uc,0},\dots,\bbar'_{\Uc,k-1})$, $\neg E_i(\bbar_{\Uc,0}c_\Uc,\bbar'_{\Uc,m}c'_\Uc)$ for all $m < k$. Once again, since these are the only $E_i$-classes containing tuples extending $c_j$ and $c'_j$, we have that $\neg F(c_\Uc,c'_\Uc)$. Therefore $F(z,w)$ is co-type-definable as well.
  \end{sloppypar}
  Since $F(z,w)$ is type-definable and co-type-definable, we can find an formula $E_{i+1}(y,z)$ that is $\{0,1\}$-valued and agrees with $F(z,w)$ on realizations of $q(y)$. Let $p_{i+1}(z) = q(z)$, $A_{i+1}=A_i$, and $n_{i+1} = 1$.

  By induction, since $n_{i+1} < n_i$ at each stage, we eventually arrive at the required $p(x)$, $E(x,y)$, and $A$.

  To see that 2 implies 3, assume 2 and pass to a model $M$ containing $A$ with $|M| \leq |T|$. By the pigeonhole principle, there is some completion of $p(x)$ over $M$ with infinitely many $E$-classes.

  Finally, assume that 3 holds. Let $a$ and $b$ be two realizations of $p(x)$ such that $E(a,b) = 0$. Fix an $M$-coheir $q(x)\supset p(x)$ and let $c \models q \res M ab$. Since $E$ is an equivalence relation, it must be the case that $E(c,a) = E(c,b) = 0$. Therefore any Morley sequence generated by $q$ is pairwise $E$-inequivalent and so $E$ has infinitely many equivalence classes among the realizations of $p$.
\end{proof}

The following corollary gives a nice test for establishing that a theory has no non-algebraic crisp imaginary types. Recall that $\dc M$ is the metric density character of $M$.

\begin{cor}\label{cor:connected-test}
  Suppose that for any $M \models T$ with $\dc M \leq |T|$ and any $1$-type $p(x)\in S_1(M)$, the set $F_p\coloneqq \{q(x,y) \in S_2(M) : q(x,y) \supseteq p(x)\cup p(y)\}$ is connected. Then $T$ has no non-algebraic crisp imaginary types.
\end{cor}
\begin{proof}
  By \cref{thm:crisp-imaginary-char-1-types}, if $T$ has a non-algebraic crisp imaginary type, then there is a countable set of parameters $A$, a $1$-type $r(x) \in S_x(A)$, and an $A$-formula $E(x,y)$ such that $E$ defines a $\{0,1\}$-valued equivalence relation on the realizations of $r(x)$ with infinitely many classes. If we find a model $M \supseteq A$ with $\dc M \leq |T|$, then there is a completion $p(x) \supseteq r(x)$ such that among the realizations of $p(x)$ there are infinitely many $E$-classes (by the pigeonhole principle). Therefore the set $F_p$ has a non-trivial clopen subset (i.e., the set of types $q(x,y) \supseteq p(x)\cup p(y)$ such that $q(x,y) \vdash E(x,y) = 1$) and is not connected. 
\end{proof}

In particular, \cref{cor:connected-test} can be used to quickly verify that the theories of infinite-dimensional Hilbert spaces, atomless probability algebras, and the Urysohn sphere do not have any non-algebraic crisp imaginary types (although for Hilbert spaces at least this was established in \cite{Ben-Yaacov2004Old}). Perhaps more surprisingly, this also implies the same for the randomization of any discrete or continuous theory.

It is unclear at the moment whether the condition in \cref{cor:connected-test} characterizes the non-existence of non-algebraic crisp imaginary types, although it seems unlikely.

\begin{quest}
  If $T$ has no non-algebraic crisp imaginary types, does it follow that for every model $M \models T$ and every $1$-type $p(x) \in S_x(M)$, the set $\{q(x,y) \in S_2(M) : q(x,y) \supseteq p(x) \cup p(y)\}$ is connected?
\end{quest}

\subsection{Relationship with earlier notions}
\label{sec:relationship-with-earlier}

Here we will show that under a non-trivial, but relatively common, niceness condition, a theory with a non-algebraic crisp imaginary type interprets the theory of an infinite discrete structure. A theory $T$ is \emph{dictionaric} if for any set of parameters $A$ and any finite tuple of variables $x$, the type space $S_{\xbar}(A)$ has a basis of definable neighborhoods. As established in \cite{CatCon}, this condition allows one to generalize many constructions from discrete logic to the context of continuous logic.

\begin{lem}\label{lem:crisp-nbd}
  If $p(x) \in S_x(A)$ is a crisp type in some sort, then there is an open neighborhood $U \ni p(x)$ such that for any $a$ and $b$ with $\tp(a/A),\tp(b/A) \in U$, either $d(a,b) \leq \frac{1}{4}$ or $d(a,b) \geq \frac{3}{4}$.
\end{lem}
\begin{proof}
  Assume that for every open neighborhood $U \ni p(x)$, there are $a$ and $b$ realizing types in $U$ such that $\frac{1}{4} < d(a,b) < \frac{3}{4}$. Then, by compactness, there are $a$ and $b$ realizing $p(x)$ such that $\frac{1}{4} \leq d(a,b) \leq \frac{3}{4}$, contradicting crispness of $p(x)$.
\end{proof}

\begin{prop}\label{prop:dict-then-interp}
  If $T$ is dictionaric and has a non-algebraic crisp type (in a real sort), then it interprets the theory of an infinite discrete structure.%
\end{prop}
\begin{proof}
  Let $p(x) \in S_x(A)$ be a non-algebraic crisp type.  By \cref{lem:crisp-nbd}, there is an open set $U \ni p(x)$ such that for any $a$ and $b$ realizing types in $U$, either $d(a,b) \leq \frac{1}{4}$ or $d(a,b) \geq \frac{3}{4}$. Since $T$ is dictionaric, we can find an $A$-definable set $D$ such that $p \in D \subseteq U$. Therefore the same fact holds for elements of $D$. The function $E(x,y) = \mxmnu{4d(x,y)-1}$ now defines a $\{0,1\}$-valued equivalence relation on the set of realizations of $D$. Since $D$ contains $p(x)$, there are infinitely many $E$-classes.
\end{proof}

Note however that it is not generally true that $T^{\mathrm{eq}}$ is dictionaric for dictionaric $T$.\footnote{See the discussion after Corollary~4.12 in \cite{CatCon}.} We will show in \cref{prop:heart-no-interpret} that without the assumption of dictionaricity, these conditions are not equivalent.

\begin{prop}\label{prop:interpreting-infinite-discrete-structure-characterization}
Fix a theory $T$.
\begin{enumerate}
\item If $T$ has formulas $\varphi(\xbar)$ and $E(\xbar,\ybar)$ such that
  \begin{itemize}
  \item $E$ is a $\{0,1\}$-valued equivalence relation on $\{\abar : \varphi(\abar) < 1\}$,
  \item $E$ has infinitely many equivalence classes, and
  \item every $E$-class of some $\abar$ with $\varphi(\abar) < 1$ contains a $\bbar$ such that $\varphi(\bbar) = 0$,
  \end{itemize}
  then $T$ has an infinite crisp imaginary sort $I$ and a bijection $f$ between $I$ and $E$-classes of elements with $\varphi < 1$ such that any automorphism of the monster fixes $a \in I$ if and only if it fixes $f(a)$.
\item Such a pair of formulas exists for any infinite crisp imaginary sort.
\item The above formulas may be taken to be restricted formulas.
\item If the above formulas are quantifier-free, they may be taken to be quantifier-free restricted formulas.
\end{enumerate}
\end{prop}
\begin{proof}
  For 1, assume that we are given $\varphi(\xbar)$ and $E(\xbar,\ybar)$ such that the bullet points hold. We may assume that $\xbar$ is an $\omega$-tuple. By replacing $\varphi(\xbar)$ with $\mxmnu{\varphi(\xbar)}$ and $E(\xbar,\ybar)$ with $\mxmnu{E(\xbar,\ybar)}$ if necessary, we may assume that $\varphi(\xbar)$ and $E(\xbar,\ybar)$ are $[0,1]$-valued everywhere. Consider the formulas
 \begin{align*}
   \psi(\xbar,\zbar) &= \mxmnu{\min( 2-3\varphi(\xbar), E(\xbar,\zbar))}, \\
   \rho(\xbar,\ybar) &= \max\left(  |\varphi(\xbar)-\varphi(\ybar)|, \sup_{\zbar}|\psi(\xbar,\zbar) - \psi(\ybar,\zbar)|  \right).
 \end{align*}
 $\rho(\xbar,\ybar)$ is the supremum of a family of pseudo-metrics and so is a pseudo-metric. Furthermore, since $\rho(\xbar,\ybar) \geq |\varphi(\xbar)-\varphi(\ybar)|$, we may regard $\varphi(\xbar)$ as a formula on the imaginary sort $H^\omega/\rho$, which we will write as $\varphi(x)$. Let $\pi : H^\omega \to H^\omega / \rho$ be the quotient map.

 We need to argue the following: The set $D = \{x \in H^\omega/\rho : \varphi(x) = 0\}$ is infinite and definable and $\rho$ is $\{0,1\}$-valued on $D$. First to see that there are infinitely many elements in $D$, note that if $\bbar$ and $\bbar'$ are two elements of $H^\omega$ in different $E$-classes with $\varphi(\bbar)=\varphi(\bbar')=0$, then $\psi(\bbar,\bbar) = 1$ and $\psi(\bbar',\bbar) = 0$, so $\rho(\bbar,\bbar') = 1$. Therefore $\pi(\bbar)$ and $\pi(\bbar')$ are distinct elements of $D$. Since there are infinitely many $E$-classes, this establishes that $D$ is infinite. To show that that $D$ is definable, fix $\e > 0$ with $\e < 1$ and note that if $\varphi(\pi(\abar)) < \e$, then $\varphi(\abar) = \varphi(\pi(\abar)) < \e$, so we can find $\bbar$ in the same $E$ class as $\abar$ such that $\varphi(\bbar) = 0$. We then have that $\varphi(\pi(\bbar)) = 0$ so $\pi(\bbar) \in D$. Now consider $\rho(\bbar,\abar)$. We have that $|\varphi(\bbar)-\varphi(\abar)| < \e$. Furthermore, for any $\cbar \in H^\omega$, we either have that $\varphi(\cbar) = 1$ or $\varphi(\cbar) < 1$. If $\varphi(\cbar) = 1$, then $\psi(\bbar,\cbar) = \psi(\abar,\cbar) = 1$. On the other hand if $\varphi(\cbar) < 1$, then $E(\bbar,\cbar)$ and $E(\abar,\cbar)$ are equal, since $E$ is a $\{0,1\}$-valued equivalence relation on $\{\xbar: \varphi(\xbar) < 1\}$ and $\bbar$ and $\abar$ are $E$-equivalent. Since we can do this for any sufficiently small $\e > 0$, we have that $D$ is definable. We can now treat $D$ as the required imaginary sort. It is clear that for any $\abar$ and $\bbar$ with $\varphi(\abar) = \varphi(\bbar) = 0$, $\pi(\abar) = \pi(\bbar)$ if and only if $E(\abar,\bbar)$, so we have that the required function $f$ exists as well.
  
 For 2, by \cite[Lem.~2.17]{CatCon}, we know that if $T$ has an infinite crisp imaginary sort, then there is a definable pseudo-metric $\rho$ on $H^\omega$ (where $H$ is the home sort) and an infinite definable set $D \subseteq H^\omega/\rho$ on which $\rho$ is $\{0,1\}$-valued. Let $\delta(x)$ be the distance predicate of $\delta$ in the imaginary sort $H^\omega/\rho$. Let $q: H^\omega \to H^\omega/\rho$ be the (definable) quotient map. Let $\varphi(\xbar) = \min(6\delta(q(\xbar)),1)$ and let $E(\xbar,\ybar) = \mxmnu{1-3\rho(\xbar,y)}$. We need to show that $\varphi(x)$ and $E(x,y)$ satisfy the bullet points in the statement of the proposition. For any $\abar$, we have that $\varphi(\abar) < 1$ if and only if $\delta(q(\abar)) < \frac{1}{6}$ (i.e., $d(q(\abar),D) < \frac{1}{6}$). For any $\abar$ and $\bbar$ satisfying $\varphi(\xbar) < 1$, we either have that $q(\abar)$ and $q(\bbar)$ have $\rho$-distance less than $\frac{1}{6}$ to the same point in $D$, in which case $\rho(\abar,\bbar) < \frac{1}{6}+\frac{1}{6} = \frac{1}{3}$ and so $E(\abar,\bbar) = 1$, or they have $\rho$-distance less than $\frac{1}{6}$ to different points in $D$, in which case $\rho(\abar,\bbar) > 1 - \frac{1}{6}-\frac{1}{6} = \frac{2}{3}$ and so $E(\abar,\bbar) = 0$. Evidently, this means that $E(\xbar,\ybar)$ is a $\{0,1\}$-valued equivalence relation on $\{\abar : \varphi(\abar) < 1\}$.

  On the other hand, if $\varphi(\abar) < 1$, then there is a $\bbar$ such that $q(\bbar) \in D$ (and therefore $\varphi(\bbar) = 0$) and $\rho(\abar, \bbar) < \frac{1}{6}$ (and therefore $E(\abar,\bbar) = 1$). Therefore the last bullet point holds as well.

  For 3 and 4, to see that we may assume that $\varphi(\xbar)$ and $E(\xbar,\ybar)$ are restricted formulas, find restricted formulas $\psi(\xbar)$ and $F(\xbar,\ybar)$ such that $|\varphi(\xbar)-\psi(\xbar)| < \frac{1}{3}$ and $|E(\xbar,\ybar)-F(\xbar,\ybar)| < \frac{1}{3}$ for all $\xbar$ and $\ybar$. (Note that we can take $\psi(\xbar)$ and $F(\xbar,\ybar)$ to be quantifier-free if $\varphi(\xbar)$ and $E(\xbar,\ybar)$ are quantifier-free.) Now let $\varphi'(\xbar)= \mxmnu{3\psi(\xbar)-1}$ and $E'(\xbar,\ybar) = \mxmnu{3E(\xbar,\ybar)-1}$. For any $\abar$, we have that if $\varphi(\abar) = 0$, then $\varphi'(\abar) = 0$ and if $\varphi'(\abar)< 1$, then $\varphi(\xbar) < \frac{2}{3}$. Furthermore, we have that for any $\abar$ and $\bbar$, if $E(\abar,\bbar) \in \{0,1\}$, then $E(\abar,\bbar) = E'(\abar,\bbar)$ as well. Therefore $\varphi'(\xbar)$ and $E'(\xbar,\ybar)$ satisfy the bullet point as required.
\end{proof}

The following corollary is more or less a rephrasing of part of \cref{prop:interpreting-infinite-discrete-structure-characterization}, but it is worth emphasizing.

\begin{cor}
  A continuous theory $T$ interprets a discrete theory with infinite models if and only if it has a pair of restricted formulas $\varphi(\xbar)$ and $E(\xbar,\ybar)$ satisfying that
   $E$ is a $\{0,1\}$-valued equivalence relation on $\{\abar : \varphi(\abar) < 1\}$,
   $E$ has infinitely many equivalence classes, and
   every $E$-class of some $\abar$ with $\varphi(\abar) < 1$ contains a $\bbar$ such that $\varphi(\bbar) = 0$.
 \end{cor}

 One thing to note about \cref{prop:dict-then-interp} is that, together with \cref{thm:crisp-imaginary-char-1-types}, it implies that if a dictionaric theory interprets an infinite discrete structure, then it has a one-dimensional interpretation of an infinite discrete structure (i.e., an infinite discrete definable subset of a quotient of the home sort). It seems unlikely that this is always true, but we do not at present know an example.

 \begin{quest}
   Does there exist a theory that interprets an infinite discrete structure but does not have a one-dimensional interpretation of an infinite discrete structure?
 \end{quest}
 
\subsection{An example separating the two notions}
\label{sec:separating}

\begin{defn}\label{defn:counterexample-structure}
  Let $\Lc_{\cx}$ be the language with a $[0,1]$-valued metric and a constant symbol $\ast$. Let $M_{\cx}$ be the $\Lc_{\cx}$-structure whose universe is $[0,1]$ with $d(x,y) = \max(x,y)$ for $x\neq y$ and $\ast^M = 0$. Let $T_\cx=\Th(M_\cx)$. 
\end{defn}

It is immediate that $T_\cx$ has a non-algebraic crisp type, namely the type of $1 \in M_\cx$ over the empty set. %

\begin{lem}\label{lem:heart-basic}\
  \begin{enumerate}
    \begin{sloppypar}
    \item An $\Lc_\heartsuit$-structure $M$ is a model of $T_\heartsuit$ if and only if $\{d(x,\ast^M):x \in M\}$ is dense in $[0,1]$ and for any distinct $a$ and $b$ in $M$, $d(a,b) = \max(d(a,\ast^M),d(b,\ast^M))$.
    \end{sloppypar}
  \item Every non-algebraic type over any set of parameters is minimal. In particular, $T_\heartsuit$ is weakly minimal and therefore superstable.
  \item $T_\cx$ has quantifier elimination.
  \item Any $n$-type $p(\xbar) \in S(A)$ is uniquely determined by the values $d(x_i,\ast)$ for $i<n$, the values of $d(x_i,x_j)$ for $i<j<n$, and the values of $d(x_i,a)$ for $i<n$ and $a\in A$.
  \item For each $n<\omega$, the type space $S_n(T_\heartsuit)$ is locally path-connected.
  \end{enumerate}
\end{lem}
\begin{proof}
  For 1, first note that for any formula $\varphi(x)$ and any closed $F \subseteq \Rb$, there is a continuous first-order theory that is satisfied by a structure $M$ if and only if $\{\varphi^M(x) : x \in M\}$ is dense in $F$. Secondly, note that the second condition is clearly satisfied by any ultrapower of $M_\heartsuit$. Therefore if $M \models T_\heartsuit$, then there is an ultrapower of $M_\heartsuit$ into which $M$ elementarily embeds, implying that for any distinct $a$ and $b$ in $M$, $d(a,b) = \max(d(a,\ast^M),d(b,\ast^M))$. Conversely, assume that $M$ is an $\Lc_\heartsuit$-structure satisfying these conditions. It's easy to see that there are winning strategies for the `approximate Ehrenfeucht-\Fraisse\ games' between $M_{\heartsuit}$ and $M$ (i.e., for a given $\e > 0$ fixed in advanced, duplicator wins if the predicates agree to within $\e$ on the finite partial map from $M$ to $M_\heartsuit$ built in the game). This is enough to establish that $M \equiv M_\heartsuit$. 

  For 2, given $A\subseteq N$, it is clear that for any $r \in (0,1]$, any permutation of $\{b \in N\setminus A : d(b,\ast) = r\}$ extends to an automorphism of $N$ fixing $A$ pointwise (namely, the automorphism that fixes the rest of $N$). This implies that $T_\heartsuit$ is weakly minimal and therefore superstable.

  For 3, note that the argument for $2$ implies that for any $n$-tuple $\abar \in N \succeq M$, $\tp(\abar/M)$ is uniquely determined by
  the indices $i<n$ for which $a_i \in M$,
  the indices $i<j<n$ for which $a_i = a_j$, and
  the values of $d(a_i,\ast)$ for each $i<n$. 
  This implies that quantifier-free types over $M$ determine full types over $M$. Since we can do this for any $M \models T_\heartsuit$, we have that $T_\heartsuit$ has quantifier elimination.

  Finally, 4 is clear from quantifier elimination and 5 follows from 4.
\end{proof}

\begin{lem}\label{lem:no-parameters}
  If $T_\heartsuit$ interprets a discrete theory with infinite models, then it does so without parameters.
\end{lem}
\begin{proof}
  By \cref{prop:interpreting-infinite-discrete-structure-characterization}, we know that if $T_\heartsuit$ interprets a discrete theory with infinite models, then it has an infinite crisp imaginary sort $I$ definable over some finite tuple $\cbar$ of parameters. Let $M\sqcup \ebar$ be a structure consisting of a model $M$ of $T_\heartsuit$ together with a tuple of new elements $\ebar$ of the same size as $\cbar$ with pairwise distance $1$. Suppose furthermore $M \sqcup \ebar$ has constants for the elements of $\ebar$. $M\sqcup \ebar$ interprets $(T_\heartsuit)_{\cbar}$ by defining a new metric $d'$ on its home sort given by the ordinary metric on $M$, $d'(x,e_i) = |d(x,\ast)-d(c_i,\ast)|$ for each $x \in M$, and $d'(e_i,e_j) = d(c_i,c_j)$. By \cite[Prop.~3.4.1, 3.4.3]{HansonThesis},\footnote{The definition of imaginary sort used in \cite{HansonThesis} is slightly more permissive than the standard definition, but it is equivalent for non-trivial theories, specifically theories whose models have more than one element in some sort \cite[Rem.~3.2.9]{HansonThesis}.} $T_\heartsuit$ interprets $\Th(M\sqcup\ebar)$ without parameters, so we have that $T_\heartsuit$ interprets $(T_\heartsuit)_{\cbar}$ for any finite tuple $\cbar$ of parameters.
\end{proof}

\begin{lem} \label{lem:acl-base} %
  If $I$ is a crisp imaginary sort of $T_\heartsuit$ definable over $\varnothing$, then for every $a \in I$, there is a unique smallest finite set $\bbar$ in the home sort such that $a \in \dcl(\bbar)$ and $\bbar \in \acl(a)$.
\end{lem}
\begin{proof}
  (This is almost identical to the proof of the analogous fact for the theory of the infinite set.)
  
  First note that by \cref{prop:interpreting-infinite-discrete-structure-characterization}, we may assume that $I$ is a definable subset of a quotient of a finite power of the home sort. For any $a \in I$, there is some finite $\bbar$ such that $a \in \dcl(\bbar)$. We may clearly assume that $\ast \notin \bbar$.

  Now assume that $\bbar\ebar$ and $\bbar\fbar$ are two different finite tuples of parameters such $\ebar \cap \fbar = \varnothing$ and $a \in \dcl(\bbar\ebar)\cap \dcl(\bbar\fbar)$. We need to show that $a \in \dcl(\bbar)$. To do this it will be sufficient to show that if $\gbar$ and $\hbar$ are two tuples of elements realizing the same type over $\bbar$, then there is an automorphism $\sigma \in \Aut(\M/\bbar)$ taking $\gbar$ to $\hbar$ and fixing $a$. Clearly we may assume that $\gbar \cap \bbar = \hbar \cap \bbar = \varnothing$.

  Let $\gbar_0 = \gbar \cap \ebar$, $\gbar_1 = \gbar \cap \fbar$, and $\gbar_2 = \gbar \setminus (\ebar\fbar)$. Let $\hbar_0\hbar_1\hbar_2$ be a partitioning of $\fbar$ satisfying that $\gbar_0\gbar_1\gbar_2 \equiv_{\bbar}\hbar_0\hbar_1\hbar_2$.  Let $\gbar' = \gbar'_0\gbar'_1\gbar'_2$ be a realization of $\tp(\gbar_0\gbar_1\gbar_2 / \bbar)$ satisfying that $\gbar'_0\gbar'_1\gbar'_2 \ind_{\bbar}\ebar\fbar\gbar\hbar$ (i.e., that any elements in common between $\gbar'$ and $\ebar\fbar\gbar$ are contained in $\bbar$). 
  
  By \cref{lem:heart-basic}, we have that $\{\gbar'_0,\gbar'_1,\gbar'_2,\ebar\gbar_0,\fbar\gbar_1,\gbar_2\}$ is pairwise independent over $\bbar$, implying that we can find an automorphism $\sigma_0 \in \Aut(\M/\bbar\ebar)$ such that $\sigma_0(\gbar_1) = \gbar'_1$, $\sigma_0$ fixes everything in  $\{\gbar'_0,\gbar'_2,\ebar\gbar_0,\gbar_2\}$, and $\{\gbar'_0,\gbar'_1\sigma_0(\fbar),\gbar'_2,\ebar\gbar_0,\gbar_2\}$ is still independent over $\bbar$. We can then find $\sigma_1 \in \Aut(\M/\bbar \cup \sigma_0(\fbar))$ such that $\sigma_1(\gbar_0)=\gbar'_0$, $\sigma_1(\gbar_2) = \gbar'_2$, and $\sigma_1$ fixes $\gbar'_1$.

  We then have that $\sigma_1(\sigma_0(\gbar)) = \gbar'$ and $\sigma_1\circ \sigma_0$ fixes $a$. By the symmetric argument, we can find an automorphism $\tau \in \Aut(\M/\bbar)$ taking $\hbar$ to $\gbar'$, so we get that $\tau^{-1}\circ \sigma_1\circ \sigma_0$ is the required automorphism.

  Since we can do this for any $\bbar\ebar$ and $\bbar\fbar$, we get that $a$ is definable over the intersection of all such finite tuples. Let $\bbar$ be this minimal tuple. Since the set $\{b_0,\dots,b_{n-1}\}$ is fixed by any automorphism fixing $a$ and since $\bbar$ is a finite tuple, we have that $\bbar \in \acl(a)$.
\end{proof}

\begin{prop}\label{prop:heart-no-interpret}
  $T_\heartsuit$ does not interpret a discrete theory with infinite models. In particular, there is a superstable theory with a non-algebraic crisp type that does not interpret a discrete theory with infinite models.
\end{prop}
\begin{proof}
  By \cref{prop:interpreting-infinite-discrete-structure-characterization} and \cref{lem:no-parameters}, if $T_\heartsuit$ interprets a discrete theory with infinite models, then there is a pair of quantifier-free restricted formulas $\varphi(\xbar)$ and $E(\xbar,\ybar)$ such that $E(\xbar,\ybar)$ is a $\{0,1\}$-valued equivalence relation on the set of tuples $\abar$ satisfying $\varphi(\abar)< 1$ with infinitely many classes such that each $E$-class contains an $\abar$ with $\varphi(\abar)  = 0$. Let $n = |\xbar|$.

  Let $I$ be the imaginary sort corresponding to $\varphi(\xbar)$ and $E(\xbar,\ybar)$. Let $\pi:H^n\to I$ be the partial function taking each tuple $\abar$ satisfying $\varphi(\abar)< 1$ to the corresponding element of $I$. Note that $\pi(\abar) \in \dcl(\abar)$. By \cref{lem:acl-base}, we have that for any $\abar$ satisfying $\varphi(\abar) < 1$, there is a unique minimal set $B_{\abar} \subseteq \abar$ such that $\pi(\abar) \in \dcl(B_{\abar})$ and $B_{\abar} \subseteq \acl(\pi(\abar))$. For any $\abar$ satisfying $\varphi(\abar) < 1$, let $I_{\abar} \subseteq n = \{0,1,\dots,n-1\}$ be the set of indices such that $a_i \in B_{\abar}$.

  Call an ordered pair $\langle \abar,\bbar \rangle$ of elements \emph{critical} if $\varphi(\abar)<1$, $\varphi(\bbar) < 1$, $\abar \equiv \bbar$, $|B_{\abar}\setminus B_{\bbar}| = 1$, and $a_i = b_i$ for all $i \in n \setminus I_{\abar}$. (Note that by \cref{lem:acl-base}, critical pairs exist.) Given a critical pair $\langle \abar,\bbar \rangle$, let $c_{\abar,\bbar}$ be the unique element of $B_{\abar} \setminus B_{\bbar}$. Note that for any critical pair $\langle \abar,\bbar \rangle$, $d(\abar,\bbar) = d(c_{\abar,\bbar},\ast)$. Also note that for any critical pair $\langle \abar,\bbar \rangle$, $E(\abar,\bbar) = 0$ (otherwise $B_{\abar}$ and $B_{\bbar}$ would not be minimal). Finally, note that if $\langle \abar,\bbar \rangle$ is a critical pair, then $\langle \bbar,\abar \rangle$ is a critical pair as well.

  Let $r = \inf\{d(\abar,\bbar):\langle \abar,\bbar \rangle~\text{a critical pair}\}$. Let $(p_k(\xbar,\ybar))_{k<\omega}$ be a sequence of $2n$-types of critical pairs satisfying $p_k(\xbar,\ybar) \vdash \varphi(\xbar) = 0 \wedge \varphi(\ybar)=0\wedge d(\xbar,\ybar) < r+2^{-k}$ for each $k<\omega$. By compactness, we may assume that the sequences is converging to some $2n$-type $q(\xbar,\ybar)$. Clearly we have that $q(\xbar,\ybar)\vdash \varphi(\xbar)=0\wedge \varphi(\ybar) = 0$. 

  $r$ cannot be $0$. To see this, assume that $r$ is $0$. This implies that $q(\xbar,\ybar) \vdash d(\xbar,\ybar) = 0$ and so $q(\xbar,\ybar) \vdash E(\xbar,\ybar) = 1$ (since $q(\xbar,\ybar)\vdash \varphi(\xbar) = 0 \wedge \varphi(\ybar)= 0$), which contradicts the fact that $E(\xbar,\ybar)$ is continuous on type space.

  \begin{sloppypar}
    For any $s \in (0,1)$, let $q_s(\xbar,\ybar)$ be the unique type satisfying $q_s(\xbar,\ybar) \vdash d(t_0,t_1) = 0$ if and only if $p(\xbar,\ybar)\vdash d(t_0,t_1) = 0$ for each pair of terms $t_0,t_1 \in \{x_0,\dots,x_{n-1},y_0,\dots,y_{n-1}\}$ and satisfying that $d(x_i,\ast)^{q_s} = s\cdot d(x_i,\ast)^p$ and $d(y_i,\ast)^{q_s}=s\cdot d(y_i,\ast)^p$ for each $i<n$ (where $\psi^q$ is the value of $\psi$ assigned by the type $q$).
  \end{sloppypar}
  Note that as $s$ approaches $1$, $q_s \to q$ topologically. Therefore, for sufficiently large $s < 1$, we have by continuity that $q_s(\xbar,\ybar)\vdash \varphi(\xbar) < 1 \wedge \varphi(\ybar) < 1$. Furthermore, for sufficiently large $s<1$, we have that $q_s(\xbar,\ybar) \vdash E(\xbar,\ybar) = 0$.

  For each $p_k(\xbar,\ybar)$, we have for each $i<j<n$ that
  \[
    p_k(\xbar,\ybar) \vdash d(x_i,y_i) =0 \vee d(x_j,y_j) =0 \vee (d(x_i,x_j) = 0 \wedge d(y_i,y_j) = 0), %
  \]
  which is a closed condition and therefore preserved under taking a limit. In particular, $q(\xbar,\ybar)$ satisfies the same thing. Furthermore, each $q_s(\xbar,\ybar)$ satisfies the same thing by construction. In particular, this implies that $q_s(\xbar,\ybar)$ is the type of a critical pair. Choose some $s < 1$ that is sufficiently large in the preceding senses. %
  Let $\abar\bbar$ be a realization of $q_s(\xbar,\ybar)$.

  By assumption, there must be some $\abar'$ satisfying $\varphi(\abar') = 0$ such that $E(\abar,\abar') = 1$. Furthermore, we can choose this $\abar'$ such that $\abar'\ind_{\abar} \bbar$ (i.e., such that $\abar' \cap \bbar \subseteq \abar$), implying that $b_{\bbar,\abar} \notin \abar'$. We need to find $\bbar'$ such that $\abar \abar' \equiv \bbar \bbar'$ and $\langle \abar',\bbar' \rangle$ is a critical pair. Define $\bbar'$ as follows: For each $i<n$, if $a'_i = c_{\abar,\bbar}$, let $b'_i = c_{\bbar,\abar}$. Otherwise, let $b'_i = a'_i$.

  By \cref{lem:heart-basic}, we have that $\abar\abar' \equiv \bbar\bbar'$ and so $\abar' \equiv \bbar'$ and therefore $\varphi(\bbar') = 0$. Furthermore, for each $i<n$, if $i \in n \setminus I_{\abar'}$, then $a'_i \neq c_{\abar,\bbar} \in C_{\abar'}$ and so $b'_i = a'_i$. Therefore we have that $\langle  \abar', \bbar' \rangle$ is a critical pair.

  Since $\langle \abar',\bbar' \rangle$ is a critical pair satisfying $\varphi(\abar') =\varphi(\bbar')=0$, we must have that $d(\abar',\bbar') = d(c_{\abar',\bbar'},\ast) \geq r$, but $d(c_{\abar',\bbar'},\ast) = sr < r$ by construction and since $C_{\abar} = C_{\abar'}$ and $C_{\bbar}=C_{\bbar'}$, we must have that $c_{\abar',\bbar'}=c_{\abar,\bbar}$, which is a contradiction.

  Therefore no such crisp imaginary sort can exist and $T_{\heartsuit}$ does not interpret a discrete theory with infinite models.
\end{proof}

\section{Generic binary predicates on $\Rb$-forests}
\label{sec:gen-bin-R-tree}%

Here we will give our example of a strictly simple theory that has no non-algebraic crisp imaginary types.

\subsection{The $\Rb$-forest for the $\Rb$-trees}
\label{sec:R-forest}%

The techniques used here are similar to those in \cite{Carlisle2020} and \cite{2021arXiv210613261H}, but neither of them are precisely applicable. Morally speaking, all of the continuous model theory results of this section are due to Carlisle and Henson in \cite{Carlisle2020}, which studied the model companion of the theory of rooted $\Rb$-trees of a given height. There the model companion was characterized as being the theory of \emph{richly branching} $\Rb$-trees. While a similar characterization of $\RF^\ast$ is possible, we won't need it for our purposes.

The results regarding $\Rb$-forests in this section are also direct adaptations of standard results in the $\Rb$-tree literature (such as those in \cite{Evans2008}), but these are often stated in terms of finite sets of points and for $\Rb$-trees in particular. Given these issues, we will give a partially self-contained presentation.

Recall than an $\Rb$-tree is a uniquely arc-connected complete metric space $(X,d)$ such that for any $x,y \in X$, the unique arc between $x$ and $y$, written $[x,y]$, is isometric to $[0,d(x,y)]\subseteq \Rb$.

\begin{defn}
  An \emph{$\Rb$-forest} is a complete extended metric space $(X,d)$ with the property that for any $x \in X$, $\{y \in X : d(x,y) < \infty\}$ is an $\Rb$-tree.
\end{defn}

For the sake of this section we will allow metrics to be extended (i.e., $[0,\infty]$-valued). Officially we could say that our metrics are $[0,1]$-valued by taking $\frac{d(x,y)}{1+d(x,y)}$ as our metric, but this is inconvenient and unnecessary to do in practice. As long as we are careful that the connectives we use with $d(x,y)$ are continuous on $[0,\infty]$, there is no issue. In particular, note that addition on $[0,\infty]$ defined in the standard way is continuous. 

In any extended metric space $X$, given $x \in X$, we will write $[x]_{<\infty}$ for the set $\{y \in X : d(x,y) < \infty\}$. %

\begin{fact}[{\cite[Thm.~3.38]{Evans2008}}]\label{fact:R-tree-embedding}
  A metric space $(X,d)$ can be isometrically embedded into an $\Rb$-tree if and only if it satisfies the \emph{4-point condition}: For all $x,y,z,w \in X$, $d(x,y)+d(z,w) \leq \max(d(x,z)+d(y,w),d(y,z)+d(x,w))$.
\end{fact}

Let $\fpc(x,y,z,w)$ denote the closed condition $d(x,y)+d(z,w)\leq \max(d(x,z)+d(y,w),d(y,z)+d(x,w))$.

\begin{prop}\label{prop:RF-axiomatizable}
  The class of $\Rb$-forests is axiomatizable.
\end{prop}
\begin{proof}
  By the continuous version of the Keisler-Shelah theorem \cite[Cor.~C.5]{Goldbring2020-GOLCSP-3}, we know that in order to establish that the class of $\Rb$-forests is axiomatizable, it is sufficient to show that any ultraproduct of $\Rb$-forests is an $\Rb$-forest and that for any complete extended metric space $X$ and any ultrafilter $\Uc$, if the ultrapower $X^{\Uc}$ is an $\Rb$-forest, then $X$ is an $\Rb$-forest as well.

  First let $(X_i)_{i \in I}$ be a family of $\Rb$-forests and let $\Uc$ be an ultrafilter on $I$. Let $X_\Uc$ be the ultraproduct of the $X_i$'s with regards to $\Uc$. We need to show that for each $x \in X_{\Uc}$, $[x]_{<\infty}$ is an $\Rb$-tree. Suppose that for some distinct $x,y \in X_\Uc$, $d(x,y) < \infty$. Let $x=(x_i)_{i\in I}/\Uc$ and $y = (y_i)_{i \in I}/\Uc$. We have that for a large set of indices, $d(x_i,y_i) < \infty$ in $X_i$. For any such index and any $r \in [0,d(x,y)]$, let $z_{i,r}$ be the unique element of $[x_i,y_i]$ satisfying $d(x_i,z_{i,r})=\frac{r}{d(x,y)}\cdot d(x_i,y_i)$. For each $r$, let $z_r = (z_i,r)_{i \in I}/\Uc$. It is straightforward to show that the function $r \mapsto z_r$ is an isometric embedding of $[0,d(x,y)]$ into $X$ satisfying $z_0 = x$ and $z_{d(x,y)} = y$. Since we can do this for any $x,y \in X$ with $d(x,y) < \infty$, we have that each $[x]_{<\infty}$ is arc-connected.

  To finish showing that $X$ is an $\Rb$-forest, we need to establish that each $[x]_{< \infty} \subseteq X$ is uniquely arc-connected. By \cref{fact:R-tree-embedding}, for each $r \in (0,\infty)$, every $\Rb$-forest satisfies the following truncated form of the 4-point condition:
  \[
    \forall x_0x_1x_2x_3 \left(\bigwedge_{i<j<4} d(x_i,x_j)< r \to \fpc(x_0,x_1,x_2,x_3)\right).
  \]
These conditions are expressible as closed conditions in continuous logic, so $X_{\Uc}$ satisfies these as well. Since this holds for each $r \in (0,\infty)$, we have that for any $x \in X_{\Uc}$, $[x]_{<\infty}$ can be embedded in an $\Rb$-tree. This implies that for any $x,y \in X$ with $d(x,y) < \infty$, there is at most one arc between $x$ and $y$. Since we have already established that there is an arc between $x$ and $y$, we have that each $[x]_{<\infty} \subseteq X_\Uc$ is uniquely arc-connected. Therefore $X_\Uc$ is an $\Rb$-forest.

Now fix a complete extended metric space $X$ and an ultrafilter $\Uc$ and suppose that $X^{\Uc}$ is an $\Rb$-forest. Identify $X$ with its image in $X^{\Uc}$. Clearly we have that for each $x \in X$, $[x]_{<\infty}$ embeds into an $\Rb$-tree, so we just need to show that $X$ is arc connected. Fix $x,y \in X$ with $d(x,y)$ and consider $[x,y] \subseteq X^{\Uc}$. For all ultrapowers $Y$ of $X^{\Uc}$, $Y$ is an $\Rb$-forest by the above. Therefore for any automorphism of any such $Y$ fixing $x$ and $y$, the set $[x,y] \subseteq X^{\Uc}$ is fixed pointwise. Therefore $[x,y] \subseteq \dcl\{x,y\}$ in $X^{\Uc}$. Since $Y$ is an elementary submodel containing $x$ and $y$, it therefore must be the case that $[x,y] \subseteq Y$ as well. Therefore $X$ is arc-connected. Since we can do this for any $x,y \in X$ with $d(x,y) < \infty$, we have that $X$ is an $\Rb$-forest.
\end{proof}

\begin{defn}
  Let $\RF$ denote the common theory of all $\Rb$-forests.
\end{defn}

\begin{prop}\label{prop:projection}
  For any $\Rb$-forests $X \subseteq Y$ and any $y \in Y$, if $d(y,X) < \infty$, then there is a unique $x \in X$ such that $d(y,x) = d(y,X)$.
\end{prop}
\begin{proof}
  Fix $z \in X$ with finite distance to $y$. By compactness of $[z,y]$, there is a unique $x_z \in [z,y]\cap X$ of minimal distance to $y$. To finish the proposition, all we need to do is show that $x_{z_0}=x_{z_1}$ for any $z_0,z_1 \in X$ with finite distance to $Y$. By \cite[Lem.~3.20,~3.22]{Evans2008}, for any such $z_0$ and $z_1$, there is a unique $w \in Y$ such that $[x_{z_0},y]\cap[x_{z_1},y] = [w,y]$ which furthermore satisfies that $w \in [x_{z_0},x_{z_1}]$.  Therefore $w \in X$, which implies that $x_{z_0} \in [w,y]$ and $x_{z_1} \in [w,y]$, but this is only possible if $x_{z_0}=x_{z_1}=w$.
\end{proof}

\begin{defn}
  Given $\Rb$-forests $X \subseteq Y$, we write $\pi_X$ for the partially defined function that takes an element of $Y$ to the unique closest point in $X$ if it exists.
\end{defn}

\begin{prop}\label{prop:RF-JEP-AP}
  $\Mod(\RF)$ has the joint embedding property and the amalgamation property.
\end{prop}
\begin{proof}
  Given two $\Rb$-forests $X$ and $Y$, the space $X \sqcup Y$ (where all new distances are set to $\infty$) is clearly an $\Rb$-forest. Therefore $\Mod(\RF)$ has the joint embedding property.

  To see that $\Mod(\RF)$ has the amalgamation property, fix an $\Rb$-forest $X$ and two $\Rb$-forest extensions $Y_0 \supseteq X$ and $Y_1\supseteq X$. We may assume that $Y_0 \setminus X$ and $Y_1\setminus X$ are disjoint. For any $y_0 \in Y_0$ and $y_1 \in Y_1$, let $\rho(y_0,y_1) = \inf_{x_0,x_1 \in X}d(y_0,x_0)+d(x_0,x_1)+d(x_1,y_1)$. Clearly if either $d(y,X_0) = \infty$ or $d(y,X_1) = \infty$, then $\rho(y_0,y_1) = \infty$. By \cref{prop:projection}, we have that if $d(y_0,X) < \infty$ and $d(y_1,X) < \infty$, then
  \[
    \rho(y_0,y_1)\leq d(y_0,\pi_X(y_0))+d(\pi_X(y_0),\pi_X(y_1))+d(\pi_X(y_1),y_1),
  \]
  but by \cite[Lem.~3.20,~3.22]{Evans2008} $\rho(y_0,y_1)$ is actually equal to the second quantity. This implies that the unique metric on $Y_0\cup Y_1$ extending the existing metrics and $\rho$ is a metric and that the resulting space is an $\Rb$-forest.
\end{proof}

The following \cref{fact:preservations} is a direct generalization of some classic preservation results, and the proofs are essentially the same (see for example Theorem 6.5.9 and Exercise 9 in Section 6.5 in \cite{hodges_1993}). A continuous theory $T$ has an \emph{$\forall \exists$-axiomatization} if it can be axiomatized by conditions of the form $\sup_x \inf_y \varphi \leq 0$ with $\varphi$ quantifier-free. %

\begin{fact}\label{fact:preservations}
  Fix a continuous theory $T$.
  \begin{enumerate}
  \item The following are equivalent:
    \begin{enumerate}
    \item $T$ has an $\forall \exists$-axiomatization.
    \item For any chain $(M_i)_{i < \alpha}$ of models of $T$, $\overline{\bigcup_{i < \alpha}M_i} \models T$.
    \end{enumerate}
  \item Suppose that for any models $A,B,C\models T$ with $B,C \subseteq A$ and $B\cap C \neq \varnothing$, $B\cap C \models T$. Then the following hold:
    \begin{enumerate}
    \item If $A \models T$ and $(B_i)_{i \in I}$ is a family of substructures of $A$ such $B_i \models T$ for each $i \in I$ and $\bigcap_{i\in I}B_i$ is non-empty, then $\bigcap_{i \in I}B_i \models T$.
    \item $T$ has an $\forall\exists$-axiomatization.
    \end{enumerate}
  \end{enumerate}
\end{fact}

\begin{prop}\ \label{prop:R-forest-intersection-AE}
  \begin{enumerate}
  \item Given an $\Rb$-forest $Y$ and a family $(X_i)_{i \in I}$ of substructures of $Y$ that are themselves $\Rb$-forests, $\bigcap_{i \in I}X_i$ is an $\Rb$-forest.
  \item For any $\Rb$-forest $Y$ and any $X_0 \subseteq Y$, there is a unique smallest $\Rb$-forest $X\subseteq Y$ such that $X_0 \subseteq X$.
  \item $\RF$ has an $\forall \exists$-axiomatization.
  \end{enumerate}
\end{prop}
\begin{proof}
  1 is immediate from the definition of $\Rb$-forest. 2 follows by considering the intersection of all $\Rb$-forests $Z\subseteq Y$ such that $X_0 \subseteq Z$. 3 follows from 1 and \cref{fact:preservations}. 
\end{proof}

\begin{defn}
  Given an $\Rb$-forest $Y$ and a set $X_0 \subseteq Y$, we will write $\langle X_0 \rangle$ for the unique smallest $\Rb$-forest $X \subseteq Y$ such that $X_0 \subseteq X$.

  Call an $\Rb$-forest $X$ \emph{finitely generated} if there is a finite set $X_0 \subseteq X$ such that $X = \langle X_0 \rangle$.
\end{defn}

Technically in the above definition we ought to specify the ambient space in our notation, but in \cref{prop:closure-isometry-type} we will see that this is unnecessary.

\begin{lem}\label{lem:closure-closure}
  For any $\Rb$-forest $A$ and set $B \subseteq A$,
  \[
    \langle B \rangle = \overline{\bigcup \{[b_0,b_1] : b_0,b_1 \in B,~d(b_0,b_1) < \infty\}}.
  \]
\end{lem}
\begin{proof}
  First we will establish the statement of the lemma in the special case that $B$ is finite. We will show this by induction. If $|B|\leq 2$, then it is clearly true, so assume that for some $n \geq 2$, the lemma holds for all $B$ with $|B| \leq n$. Let $B$ be a subset of $A$ with $|B| = n$ and fix some $b \in B$. If $d(b,B) = \infty$, then $\langle Bb \rangle = \langle B \rangle \cup \{b\}$, so we are done. Otherwise, if $d(b,B) < \infty$, let $c \in B$ be an element with $d(b,c) < \infty$. We now have that $\langle B \rangle \cup [c,b]$ is the unique minimal $\Rb$-forest containing $Bb$.

  \begin{sloppypar}
    Now to finish the argument we will show that for any $B \subseteq A$, $\langle B \rangle = \overline{\bigcup \{\langle B_0 \rangle : B_0 \subseteq B,~B_0~\textup{finite}\}}$. Let $C = \bigcup\{\langle B_0 \rangle : B_0\subseteq B,~B_0~\textup{finite}\}$. Clearly $B \subseteq C \subseteq \langle B \rangle$. By \cref{fact:preservations} and \cref{prop:R-forest-intersection-AE}, $\overline{C}$ is an $\Rb$-forest. Therefore $\overline{C} = \langle B \rangle$.
  \end{sloppypar}
\end{proof}

\begin{prop}\label{prop:closure-isometry-type}
  For any $\Rb$-forests $A$ and $B$ and any tuples $\abar \in A$ and $\bbar \in B$ of the same (possibly infinite) length $\alpha$, if $d(a_i,a_j) = d(b_i,b_j)$ for all $i<j<\alpha$, then there is a unique isometry $f: \langle \abar \rangle \to \langle \bbar \rangle$ satisfying $f(a_i) = b_i$ for each $i<\alpha$.
\end{prop}
\begin{proof}
  It is a standard fact and not hard to show that for any subset $A$ of an $\Rb$-tree, $\langle A \rangle$ is the injective envelope\footnote{Also known as the metric envelope, tight span, or hyperconvex hull.} of $A$, written $e(A)$. By \cite[Thm.~2.1]{Isbell1964}, for any metric spaces $X$ and $Y$ with an isometry $f : X \to Y$, there is a unique isometry from $e(X)$ to $e(Y)$ extending $f$. Since $\langle  \abar \rangle = \bigcup_{i < |\abar|}\langle \abar \cap [a_i]_{<\infty} \rangle$, the full result follows immediately.
\end{proof}

\begin{defn}
  Call an $\Rb$-forest $X$ \emph{type--existentially closed} or \emph{t.e.c.}\ if for any finitely generated $\Rb$-forest $Y \subseteq X$ and any finitely generated $\Rb$-forest $Z \supseteq Y$, there is an isometric embedding $f : Z \to X$ fixing $Y$ pointwise.
\end{defn}

\begin{prop}\label{prop:model-companion-QE}\
  \begin{enumerate}
  \item\label{mc-1} Every $\Rb$-forest $A$ has an $\Rb$-forest extension $B \supseteq A$ that is t.e.c.
  \item\label{mc-2} If $A$ and $B$ are t.e.c.\ $\Rb$-forests and $\abar \in A$ and $\bbar \in B$ are finite $n$-tuples such that $d(a_i,a_j) = d(b_i,b_j)$ for each $i<j<n$, then $(A,\abar)$ and $(B,\bbar)$ are back-and-forth equivalent and therefore $(A,\abar) \equiv (B,\bbar)$. In particular, there is a common theory of all t.e.c.\ $\Rb$-forests.
  \item\label{mc-3} Any ultrapower of a t.e.c.\ $\Rb$-forest is t.e.c.
  \item\label{mc-4} The common theory of t.e.c.\ $\Rb$-forests has quantifier-elimination and is the model companion of $\RF$.
  \end{enumerate}
\end{prop}
\begin{proof}
  For \ref{mc-1}, first note that by \cref{fact:preservations} and \cref{prop:R-forest-intersection-AE}, closures of unions of chains of $\Rb$-forests are $\Rb$-forests. By \cref{prop:RF-JEP-AP}, this means that given an $\Rb$-forest $A$, we can build a chain $(A_i)_{i<|A|}$ of $\Rb$-forests with the property that any finitely generated extension $C$ of any finitely generated substructure $B$ of $A$ has an embedding into $\overline{\bigcup_{i<|A|}A_i}$ that fixes $B$. Repeating this, we can build a chain $(A'_j)_{j<\omega_1}$ of $\Rb$-forests extending $A$ with the property that $A'_{j+1}$ embeds any finitely generated substructure of a finitely generated substructure of $A'_j$. The union is then metrically closed (since $\omega_1$ has uncountable cofinality) and so is a t.e.c.\ extension of $A$.

  For \ref{mc-2}, fix $(A,\abar)$  and $(B,\bbar)$ as in the statement of the proposition. By \cref{prop:closure-isometry-type}, $\langle \abar \rangle$ and $\langle \bbar \rangle$ are isometric by an isometry $f$ satisfying $f(a_i) = b_i$ for each $i<|\abar|$. Back-and-forth equivalence is immediate from the definition of type--existential closure and \cref{prop:RF-JEP-AP}. Finally the fact that back-and-forth implies elementary equivalence is proven in the same manner as it is in discrete logic.

  For \ref{mc-3}, fix an $\Rb$-forest $A$ and an ultrafilter $\Uc$ on an index set $I$. %
  Fix a finite $n$-tuple $\abar \in A^\Uc$ and assume that $d(a_j,a_k)<\infty$ for each $j<k<n$. Let $(\abar^i)_{i\in I}$ be a family of tuples in $A$ corresponding to $\abar$. We may assume without loss of generality that for each $i\in I$, $d(a_j^i,a_k^i) < \infty$ for each $j<k<\omega$. Let $B$ be a finitely generated $\Rb$-tree extending $\langle \abar \rangle$. By \cref{prop:projection}, we know that $B \setminus \langle \abar \rangle$ has finitely many connected components $B_0,\dots,B_{m-1}$ with the property that for each $j<m$, $\overline{B_j}$ is $B_j$ together with a single point $c_j \in \langle \abar \rangle$. For each $c_j$, we can find $k(j,0)$ and $k(j,1)$ less than $n$ such that $c_j \in [a_{k(j,0)},a_{j(j,1)}]$. Let $r_j \in [0,1]$ be the unique value such that $d(a_{k(j,0)},c_j) = r_j \cdot d(a_{k(j,1)},c_j)$.

  For each $i \in I$, we can find points $c^i_j$ such that  $d(a_{k(j,0)}^i,c_j^i) = r_j \cdot d(a_{k(j,1)}^i,c_j^i)$. By the amalgamation property and the fact that $A$ is t.e.c., we can find $B_0^i,\dots,\allowbreak B_{m-1}^i$ disjoint from $\langle \abar^i \rangle$ such that for each $j<m$, $B_j^i\cup \{c_j^i\}$ is an $\Rb$-tree. Let $\bbar$ be a finite tuple in $B$ such that $\langle \abar\bbar \rangle = B$ and $\bbar \cap \langle \abar \rangle = \varnothing$. For $i \in I$, let $\bbar^i$ be the images of $\bbar$ under the isometries between $B_j$ and $B_j^i$. Finally let $\bbar^\Uc$ be the element of $A^\Uc$ corresponding to the family $(\bbar^i)_{i \in I}$. By construction we have that $\abar\bbar$ and $\abar\bbar^\Uc$ have the same pairwise distances, so by \cref{prop:closure-isometry-type}, there is an isometry taking $B=\langle \abar\bbar \rangle$ to $\abar\bbar^\Uc$ that fixes every element of $\langle \abar \rangle$.

  Now finally for an arbitrary tuple $\abar$ perform the above construction for each finite distance class of $\abar$ and for any connected component $C$ of an extension $B \supseteq \langle \abar \rangle$ not connected to $\langle \abar \rangle$, embed $C$ into some finite distance class of $A$ not used elsewhere in the embedding. Since we can do this for any finite $\abar$ and any finitely generated $B \supseteq \langle \bbar \rangle$, we have that $A^\Uc$ is t.e.c.

  For \ref{mc-4}, note that by \ref{mc-2} and \cref{prop:closure-isometry-type}, we have that the type of any tuple $\abar$ in a t.e.c.\ $\Rb$-forest is determined by its quantifier-free type. By \ref{mc-3} this happens in some $\aleph_0$-saturated model of the theory. Therefore the theory has quantifier elimination. Since this theory has quantifier elimination, it is the model companion of $\RF$ by \ref{mc-1}.
\end{proof}

\begin{defn}
  Let $\RFs$ denote the model companion of $\RF$.
\end{defn}

Recall that in a type space $S_1(A)$, the induced metric on $S_1(A)$, written $\partial$, is defined by $\partial(p,q) = \inf\{d(a,b) : a \models p,~b\models q\}$. This concept is still well defined when we allow metric structures with extended metrics.

\begin{prop}\label{prop:RF-stable}
  $\RFs$ is strictly stable.
\end{prop}
\begin{proof}
  By Propositions~\ref{prop:projection} and \ref{prop:model-companion-QE}, we know that given any model $A$ of $\RFs$ we know that there is a unique type in $S_1(A)$ of an element with infinite distance to any element of $A$ and for any $b$ in the monster with finite distance to $A$, $\tp(b/A)$ is uniquely determined by $\pi_A(b)$ and $d(b,A)$. This immediately gives that $\RFs$ is stable. To see that it is strictly stable, note that for any $b_0,b_1$ both with finite distance to $A$, we have that $\partial(\tp(b_0/A),\tp(b_1/A)) = |d(b_0,A)-d(b_1,A)|$ if $\pi_A(b_0) = \pi_A(b_1)$ and $\partial(\tp(b_0/A),\tp(b_1/A)) = d(b_0,A) + d(\pi_A(b_0),\pi_A(b_1)) + d(b_1,A) = d(b_0,b_1)$ if $\pi_A(b_0) \neq \pi_A(b_1)$. This immediately allows us to find for any infinite $\kappa$ a model $A$ of $\RFs$ and a family $(b_i)_{i<\kappa^{\aleph_0}}$ of elements such that $\partial(\tp(b_i/A),\tp(b_j/A)) \geq 1$ for each $i<j<\kappa^{\aleph_0}$, implying that $\RFs$ is not superstable.
\end{proof}

\begin{prop}\label{prop:RF-ind}
  For any sets $A$, $B$, and $C$ of parameters, $B \ind_A C$ if and only if $\langle AB \rangle\cap \langle AC \rangle = \langle A \rangle$. 
\end{prop}
\begin{proof}
  Given \cref{prop:RF-JEP-AP}, we clearly have that if $B \ind_A C$, then $\langle ABC \rangle$ needs to be the free amalgamation of $\langle AB \rangle$ and $\langle AC \rangle$ over $\langle A \rangle$, implying that $\langle AB \rangle\cap \langle AC \rangle = \langle A \rangle$. By quantifier elimination and stationarity of independence, this implies the converse as well.
\end{proof}

\subsection{Generic \oo-Lipschitz binary predicates}
\label{sec:generic-11}

\begin{defn}
  A function $f: X^2 \to X$ is \emph{\oo-Lipschitz} if for any $a \in X$, $x \mapsto f(a,x)$ and $x \mapsto f(x,a)$ are $1$-Lipschitz functions.
\end{defn}

Note that this condition is equivalent to being $1$-Lipschitz in the $\ell^1$-metric (i.e., $d(x_0x_1,y_0y_1) = d(x_0,y_0) + d(x_1,y_1)$), which is strictly weaker than being $1$-Lipschitz in the max metric. Clearly any \oo-Lipschitz function is $2$-Lipschitz with regards to the max metric. A fact which we will use several times is that if $(f_i)_{i\in I}$ is any family of \oo-Lipschitz $[0,1]$-valued functions on a metric space, then $\sup_{i \in I}f_i(x,y)$ and $\inf_{i \in I}f_i(x,y)$ are also \oo-Lipschitz and $[0,1]$-valued.

\begin{defn}
  Let $\Lc_{\RFR}$ be the language containing an extended metric and a single $2$-Lipschitz $[0,1]$-valued binary predicate $R$.

  Let $\RFR$ be the theory that says that $(A,d) \models \RF$ and $R$ is \oo-Lipschitz. 
\end{defn}

Note that by \cref{prop:R-forest-intersection-AE}, $\RFR$ has an $\forall\exists$-axiomatization. 

The proof of the following proposition makes it clear that we need the \oo-Lipschitz condition to ensure that the class of models of $\RFR$ has the amalgamation property.

\begin{prop}\label{prop:RFB-JEP-AP}
  $\Mod(\RFR)$ has the joint embedding property and the amalgamation property.
\end{prop}
\begin{proof}
  Now suppose that $A$, $B_0$, and $B_1$ are models of $\RFR$ with $A \subseteq B_0$ and $A \subseteq B_1$. Let $C$ be the free amalgamation of the underlying $\Rb$-forests of $B_0$ and $B_1$ over $A$. We just need to show that we can extend $R$ to a \oo-Lipschitz function on $C$. Let
  \[
    Q(x,y) = \mxmnu{\inf\{R(z,w)+d(x,z)+d(y,w):R(z,w)~\text{defined}\}}.
  \]
  Since the function $( x,y )\mapsto R(z,w)+d(x,z)+d(y,w)$ is \oo-Lipschitz for each fixed $z$ and $w$, we immediately have that $Q(x,y)$ is \oo-Lipschitz on $C$. It is also $[0,1]$-valued by construction. We just need to argue that for any pair $( x,y )$, if $R(x,y)$ is defined, then $Q(x,y) = R(x,y)$. To do this we just need to show that $R(x,y) \leq R(z,w) + d(y,w) + d(z,w)$ for each $( z,w )$ at which $R$ is defined. If $\{x,y,z,w\} \subseteq B_0$ or $\{x,y,z,w \} \subseteq B_1$, then this is immediate from the fact that $R$ is \oo-Lipschitz, so assume that $\{x,y\} \subseteq B_0$ and $\{z,w\}\subseteq B_1$. If any of the elements of $\{x,y,z,w\}$ have infinite distance to $A$, then the inequality is trivial, so assume that all four of them have finite distances. We now have that $d(x,z) = d(x,\pi_A(x))+d(\pi_A(x),\pi_A(z)) + d(\pi_A(z),z)$ and likewise for $y$ and $w$. Finally, we have that
  \begin{align*}
    |R(x,y) - R(\pi_A(x),\pi_A(y))| &\leq d(x,\pi_A(x)) + d(y,\pi_A(y)),\\
    |R(\pi_A(x),\pi_A(y))- R(\pi_A(z),\pi_A(w))| &\leq  d(\pi_A(x),\pi_A(z)) + d(\pi_A(y),\pi_A(w)), \\
    |R(\pi_A(z),\pi_A(w)) - R(z,w)| &\leq d(\pi_A(z),z) + d(\pi_A(w),w),
  \end{align*}
  whereby $|R(x,y) - R(z,w)|$ is less than or equal to the sum of the three quantities on the right, which is precisely $d(x,z) + d(y,w)$. Therefore $R(x,y) \leq R(z,w) + d(x,z) + d(y,w)$, as required.

  The joint embedding property follows by amalgamating over the empty model of $\RFR$.
\end{proof}

\begin{defn}
  Call a model $A$ of $\RFR$ \emph{type--existentially closed} or \emph{t.e.c.}\ if for any finitely generated $B \subseteq A$ with $B \models \RFR$ and any finitely generated $C \supseteq B$ with $C \models \RFR$, there is an $\Lc_{\RFR}$-isomorphism $f : C \to A$ fixing $B$ pointwise.
\end{defn}

\begin{prop}\label{prop:RFB-model-companion-tec-basic}\ 
  \begin{enumerate}
  \item\label{mcB-1} Any model of $\RFR$ has an extension that is a t.e.c.\ model of $\RFR$.
  \item\label{mcB-2} If $A$ and $B$ are t.e.c.\ models of $\RFR$ and $\abar \in A$ and $\bbar \in B$ are finite $n$-tuples such that there is an $\Lc_{\RFR}$-isomorphism $f: \langle \abar \rangle \to \langle \bbar \rangle$ satisfying $f(a_i) = b_i$ for each $i<n$, then $(A,\abar)$ and $(B,\bbar)$ are back-and-forth equivalent and therefore $(A,\abar)\equiv (B,\bbar)$. In particular, there is a common theory of all t.e.c.\ models of $\RFR$.
  \end{enumerate}
\end{prop}
\begin{proof}
  \ref{mcB-1} and \ref{mcB-2} follow by the same argument as parts \ref{mc-1} and \ref{mc-2} of \cref{prop:model-companion-QE}, using \cref{prop:RFB-JEP-AP}.
\end{proof}

\begin{defn}
  Given two sets $A$ and $B$, a \emph{correlation between $A$ and $B$} is a total, surjective relation $O \subseteq A\times B$ (i.e., a relation such that $(\forall a \in A)(\exists b \in B)O(a,b)$ and $(\forall b \in B)(\exists a \in A)O(a,b)$). Given sets $A$ and $B$ and $n$-tuples $\abar \in A$ and $\bbar \in B$, we'll write $\corr(A,\abar;B,\bbar)$ for the set of correlations $O$ between $A$ and $B$ satisfying that for each $i<n$, $O(a_i,b_i)$. We will write $\corr(A,B)$ for $\corr(A,\varnothing;B,\varnothing)$.

  Given two extended metric structures $A$ and $B$ and a correlation $O \in \corr(A,B)$, the \emph{$K$-truncated $d$-distortion\footnote{Distortion for metrics is normally defined with a factor of $\frac{1}{2}$, but we will not need this here.} of $O$} is the quantity 
  \[
    \dis_d^K(O)=\sup_{O(a,b),O(a',b')}\left|\mxmn{d(a,a')}{0}{K}-\mxmn{d(b,b')}{0}{K}\right|.
  \]
  Given two $\Lc_{\RFR}$-structures $A$ and $B$ and a correlation $O \in \corr(A,B)$, the \emph{$K$-truncated distortion of $O$} is the quantity
  \[
    \dis^K(O)= \max\left(\dis_d^K(O),\sup_{O(a,b),O(a',b')}|R(a,a')-R(b,b')|\right).
  \]
  Given tuples $\abar \in A$ and $\bbar \in B$ of the same length, the \emph{$K$-truncated $\Lc_{\RFR}$-distance between $(A,\abar)$ and $(B,\bbar)$} is the quantity
  \[
    \rKRFR(A,B) = \inf_{O \in \corr(A,\abar;B,\bbar)}\dis^K(A,B).
  \]
  Given an $n$-tuple $\abar$ in an $\Rb$-forest, we write $\langle \abar \rangle_K$ for the set $\bigcup\{[a_j,a_k]:j,k<n,~d(a_j,a_k) < K\}$. We will say that $K$ is \emph{large for $\abar$} if $K \geq 1$ and for any $j,k<n$, $d(a_j,a_k) < \infty$ implies that $d(a_j,a_k) < K$.
\end{defn}

Note that for any finite tuple $\abar$ in an $\Rb$-forest, there is a $K$ such that $\langle \abar \rangle_K = \langle \abar \rangle$.

\begin{defn}  
  Given $a,b \in A \models \RFR$ with $d(a,b) < \infty$ and $r \in [0,1]$, we will write $\interp{a}{r}{b}$ for the unique element of $[a,b]$ satisfying $d(a,\interp{a}{r}{b}) = r\cdot d(a,b)$.
\end{defn}

Recall that by \cref{lem:closure-closure}, for any tuple $\abar$, elements of the form $\interp{a_j}{r}{a_k}$ cover $\langle \abar \rangle$. Also, note that for any $K > 0$ and $r\in [0,1]$, $(a,b) \mapsto \interp{a}{r}{b}$ is a definable function on the definable set\footnote{This is not in general a definable set in an arbitrary metric structure but it is always definable in $\Rb$-forests with a distance predicate of $\max(d(x,y)-K,0)$.} of pairs $(a,b)$ with $d(a,b) \leq K$.

\begin{lem}\label{lem:metric-dis-extension} Fix extended metric spaces $A$ and $B$.
  \begin{enumerate}
  \item\label{met-ext-1} For any subset $A_0\subseteq A$, correlation $O \in \corr(A_0,B)$, and $1$-Lipschitz function $f : B \to [0,1]$, there is a $1$-Lipschitz function $g : A \to [0,1]$ such that for any $(a,b) \in O$ and any $K \geq 1$, $|g(a)-f(b)|\leq \dis^K_d(O)$.
  \item\label{met-ext-2} For any subset $A_0 \subseteq A$ and 1-Lipschitz functions $f: A_0 \to [0,1]$ and $g:A \to [0,1]$, there is a $1$-Lipschitz function $h:A \to [0,1]$ extending $f$ such that $\sup_{a \in A}|h(a)-g(a)| = \sup_{a \in A_0}|f(a)-g(a)|$.
  \item\label{met-ext-3} For any $O \in \corr(A,B)$, let $O^{(2)} = \{((a,a'),(b,b')) : (a,b),(a',b') \in O\}$ and let $d^{+2}$ be the metric $d^{+2}(xy,zw) = d(x,z) + d(y,w)$. For any $K$, $\dis^K_{d^{+2}}(O^{(2)}) \leq 2\dis^K_d(O)$ (where $\dis^K_{d^{+2}}$ is computed with regards to the metric $d^{+2}$).
  \end{enumerate}
\end{lem}
\begin{proof}
  For \ref{met-ext-1}, let $g(x) = \inf_{(a,b) \in O}\mxmnu{d(x,a) + f(b)}$. $g(x)$ is $1$-Lipschitz and $[0,1]$-valued by construction. We also clearly have that for any $(a,b) \in O$, $g(a) \leq f(b)$.  Fix $(a',b') \in O$. If $d(a,a') \geq 1$, then $d(a,a') + f(b') \geq 1$. Otherwise, if $d(a,a') < 1$, we have that $\mxmn{d(b,b')}{0}{K} \leq  d(a,a')+ \dis^K_d(O)$. Since $f$ is $1$-Lipschitz, it is $1$-Lipschitz with regards to the metric $\mxmn{d}{0}{K}$, so we have that $f(b') \geq f(b) - \mxmn{d(b,b')}{0}{K}$. Therefore $f(b') \geq f(b) - d(a,a') - \dis^K_d(O)$ and so $d(a,a')+f(b') \geq f(b) - \dis^K_d(O)$. Since we can do this for any $(a',b') \in O$, we have that $g(a) \geq f(b) - \dis^K_d(O)$, whence $|g(a)-f(b)| \leq \dis^K_d(O)$.

  For \ref{met-ext-2}, let $r = \sup_{a \in A_0}|f(a)-g(a)|$. Let
  \[
    h(x) = \min\left(g(x) + r,\inf_{a \in A_0}\mxmnu{d(x,a)+f(a)}\right). 
  \]
  Clearly we have that $h(a) \leq g(a) + r$ for any $a \in A$. Also note that $h(x)$ is $1$\nobreakdash-\hspace{0pt}Lipschitz and $[0,1]$-valued by construction. By the same argument as in \ref{met-ext-1}, we have that for any $a \in A_0$, $\inf_{a'\in A_0}\mxmnu{d(a,a')+f(a')} = f(a)$. By the choice of $r$, this implies that $h(x)$ extends $f(x)$. We just need to show that $h(a) \geq g(a) - r$ for each $a \in A$. For any $a' \in A_0$, we have that $f(a') \geq g(a')-r$, whence
  \begin{align*}
    \inf_{a'\in A_0}\mxmnu{d(a,a')+f(a')} &\geq \inf_{a'\in A_0}\mxmnu{d(a,a')+g(a') - r} \\
                                          & \geq \inf_{a'\in A_0}\mxmnu{d(a,a')+g(a')} - r\\
                                          & \geq \mxmnu{g(a)} - r = g(a) - r.
  \end{align*}
  Therefore, since $g(a) +r \geq g(a) - r$, we have that $h(a) \geq g(a) - r$, as required.

  \ref{met-ext-3} is immediate from the relevant definitions. %
\end{proof}

\begin{lem}\label{lem:RFB-extension}
  Fix an $n$-tuple $\abar$ and an $(n+1)$-tuple $\bbar c$ in models of $\RFR$. If $K$ is large for $\bbar c$, then for any $\e > 0$, there is a model $\langle \abar e \rangle \models \RFR$ extending $\langle \abar \rangle$ such that $d(e,\langle \abar \rangle) = \infty$ if and only if $d(c,\langle \bbar \rangle) = \infty$ and
  \[
    \rKRFR(\langle \abar e \rangle_K,\abar e; \langle \bbar c \rangle,\bbar c) \leq 4\rKRFR(\langle \abar \rangle_K,\abar;\langle \bbar \rangle,\bbar) + \e.
  \]
\end{lem}
\begin{proof}
  Let $\rho = \rKRFR(\langle \abar \rangle_K,\abar;\langle \bbar \rangle,\bbar)$. Fix $\e > 0$ and fix a correlation $O \in \corr(\langle \abar \rangle_K,\allowbreak\abar;\allowbreak\langle \bbar \rangle,\bbar)$ with $\dis^K(O) \leq \rho+\frac{1}{4}\e$. There are two cases to consider; either $d(c,\langle \bbar \rangle) = \infty$ or $d(c,\langle \bbar \rangle) < \infty$.

  If $d(c,\langle \bbar \rangle) = \infty$, then we will make $\langle \abar \rangle\cup \{e\}$ into a model of $\RFR$ by first setting $d(e,x) = \infty$ for any $x \in \langle \abar \rangle$. We then immediately have that $\dis^K_d(O\cup \{(e,c)\}) = \dis^K_d(O)$. The functions $x \mapsto R(x,c)$ and $x \mapsto R(c,x)$ are $1$-Lipschitz and $[0,1]$-valued on $\langle \bbar \rangle$. By \cref{lem:metric-dis-extension} part \ref{met-ext-1}, we can find $1$-Lipschitz function $f: \langle \abar \rangle \to [0,1]$ and $g: \langle \abar \rangle \to [0,1]$ satisfying that for any $(a,b) \in O$, $|f(a)-R(b,c)| \leq \rho$ and $|g(a)-R(c,b)|\leq \rho$. (In particular, we take $A = \langle \abar \rangle$, $A_0 = \langle \abar \rangle_K$, and $B = \langle \bbar \rangle$.) We then have that defining $R(a,e) = f(a)$, $R(e,a) = g(a)$, and $R(e,e) = R(c,c)$ gives a model of $\RFR$ satisfying $\rKRFR(\langle \abar e \rangle_K,\abar e; \langle \bbar c \rangle,\bbar c) < \rho + \e$.

  If $d(c,\langle \bbar \rangle) < \infty$, first note that $d(c,\langle \bbar \rangle) < K$. Fix some $a^\dagger \in \langle \abar \rangle_K$ such that $(a^\dagger,\pi_{\langle \bbar \rangle}(c)) \in O$. Let $b^\dagger = \pi_{\langle \bbar \rangle}(c)$. Build an $\Rb$-forest $\langle \abar e \rangle$ extending $\langle \abar \rangle$ by taking the free amalgamation that is a copy of $[0,d(b^\dagger,c)]$ glued to $\langle \abar \rangle$ by identifying $0$ with $a^\dagger$ (and letting $e$ be the point $d(b^\dagger,c)$ in $[0,d(b^\dagger,c)]$). Let $O'$ be $O\cup \{(\interp{a^\dagger}{r}{e},\interp{b^\dagger}{r}{c}) : r \in [0,1]\}$. $O'$ is a correlation between $\langle \abar e \rangle_K$ and $\langle \bbar c \rangle$. For any $a \in \langle \abar \rangle_K$, $a' \in \langle \abar e \rangle_K \setminus \langle \abar \rangle_K$, $b \in \langle \bbar \rangle$, and $b' \in \langle \bbar c \rangle\setminus \langle \bbar \rangle$ with $(a,b),(a',b') \in O'$, we have that
  \begin{align*}
    |\mxmnu{d(a,a')}-\mxmnu{d(b,b')}| & \leq |\mxmnu{d(a,a^\dagger)}-\mxmnu{d(b,b^\dagger)}| + |\mxmnu{d(a^\dagger,a')} - \mxmnu{d(b^\dagger,b')}| \\
                                      & \leq |\mxmnu{d(a,a^\dagger)}-\mxmnu{d(b,b^\dagger)}| + 0 \leq \dis^K_d(O).
  \end{align*}
  \begin{sloppypar}
    Therefore $\dis^K_d(O') \leq \dis^K_d(O)$. Since $\dis^K_d(O) \leq \dis^K_d(O')$, we have that $\dis^K_d(O') = \dis^K_d(O)$.
  \end{sloppypar}
  Let $d^{+2}$ be the metrics on $A = \langle \abar e \rangle^2$ and $B = \langle \bbar c \rangle^2$ defined by $d^{+2}(xy,zw) = d(x,z) + d(y,w)$. Let $O'^{(2)} = \{((a,a'),(b,b')) : (a,b),(a',b') \in O'\}$. By \cref{lem:metric-dis-extension} part \ref{met-ext-3}, we have that $\dis^K_{d^{+2}}(O'^{(2)}) \leq 2\dis^K_d(O') \leq 2 \rho + \e$. Recall that \oo-Lipschitz functions on $\langle \abar e \rangle$ are the same thing as $1$-Lipschitz functions on $(A,d^{+2})$. By \cref{lem:metric-dis-extension} part \ref{met-ext-1}, we can find a \oo-Lipschitz function $f: \langle \abar e \rangle^2 \to [0,1]$ satisfying that for any $(a,b),(a',b') \in O'$, $|f(a,a') - R(b,b')| \leq 2\rho + \frac{1}{2}\e$. By \cref{lem:metric-dis-extension} part \ref{met-ext-2}, we can find a \oo-Lipschitz function $Q:\langle \abar e \rangle^2$ extending $R : \langle \abar \rangle^2 \to [0,1]$ such that $\sup_{a,a' \in \langle \abar e \rangle}|Q(a,a') - f(a,a')| = \sup_{a,a' \in \langle \abar e \rangle_K}|R(a,a')-f(a,a')| \leq 2\rho + \frac{1}{2}\e$. Therefore, we now have that for any $(a,b),(a',b') \in O'$,
  \begin{align*}
    |Q(a,a')-R(b,b')| & \leq |Q(a,a')-f(a,a')| + |f(a,a')- R(b,b')| \\
    &\leq 2 \rho + \frac{1}{2}\e + 2 \rho +\frac{1}{2}\e = 4 \rho + \e.
  \end{align*}
  Since we can do this for any such $(a,b)$ and $(a',b')$, we have that $\dis^K(O') \leq \max(\dis^K_d(O'),4\rho+\e) = 4\rho+\e$. Therefore $\rKRFR(\langle \abar e \rangle_K,\abar e; \langle \bbar c \rangle,\bbar c) \leq 4\rho + \e$, as required.
\end{proof}

\begin{lem}\label{lem:rho-limit}
  For any $A \models \RFR$, index set $I$, ultrafilter $\Uc$ on $I$, and family $(\abar^i)_{i \in I}$ of $n$-tuples with limit $\abar$ in the ultrapower $A^\Uc$, if $K > 0$ is large for $\abar$, then
  \[
    \lim_{i \to \Uc}\rKRFR(\langle \abar^i \rangle_K,\abar^i;\langle \abar \rangle;\abar) = 0.
  \]
\end{lem}
\begin{proof}
  We may assume without loss of generality that for any $i \in I$ and $j,k<n$, if $d(a_j,a_k) < \infty$, then $d(a^i_j,a^i_k) < K$ and if $d(a_j,a_k) = \infty$, then $d(a^i_j,a^i_k) > 2K$.

  Given $i \in I$, let $O^i \subseteq \langle \abar^i \rangle_K \times \langle \abar \rangle$ be given by
  \[
    O^i = \{(\interp{a^i_j}{r}{a^i_k},\interp{a_j}{r}{a_k}) : j,k<n,~d(a_j,a_k) < \infty,~r \in [0,1]\}.
  \]
  By \cref{lem:closure-closure}, $O^i$ is a correlation between $\langle \abar^i \rangle_K$ and $\langle \abar \rangle$. The result now follows from the fact that $(x,y,r) \mapsto \interp{x}{r}{y}$ is uniformly definable and Lipschitz on $\{(x,y):d(x,y) \leq K\}\times [0,1]$ (with the max metric and a Lipschitz constant that depends on $K$) as well as the fact that $\abar$ is a finite tuple. In particular, for each quadruple $j,k,\ell,m<n$ with $d(a_j,a_k) < \infty$ and $d(a_\ell,a_m) < \infty$,
  \begin{align*}
    \lim_{i \to \Uc}\sup_{r,s \in [0,1]}&|\mxmn{d(\interp{a_j^i}{r}{a_k^i},\interp{a_\ell^i}{s}{a_k^i})}{0}{K} - \mxmn{d(\interp{a_j}{r}{a_k},\interp{a_\ell}{s}{a_k})}{0}{K}| = 0, \\
    \lim_{i \to \Uc} \sup_{r,s,\in [0,1]}&|R(\interp{a_j^i}{r}{a_k^i},\interp{a_\ell^i}{s}{a_k^i})-R(\interp{a_j}{r}{a_k},\interp{a_\ell}{s}{a_k})| = 0.\qedhere
  \end{align*}
\end{proof}

For part of the following proof note that if $\rKRFR(\langle \abar \rangle_K,\abar;\langle \bbar \rangle,\bbar) = 0$ for some finite tuples $\abar$ and $\bbar$, then by compactness it is actually the case that $(\langle \abar \rangle_K,\mxmn{d}{0}{K},R,\abar)$ and $(\langle \bbar \rangle,\mxmn{d}{0}{K},R,\bbar)$ (where $\mxmn{d}{0}{K}$ is the metric $\mxmn{d(x,y)}{0}{K}$) are isomorphic as pointed $\Lc_{\RFR}$-structures \cite[Cor.~1.14 (iii)]{Hansona}.

\begin{prop}\label{prop:RFB-tec-ult}
  Any ultrapower of a t.e.c.\ model of $\RFR$ is t.e.c.
\end{prop}
\begin{proof}
  Fix a t.e.c.\ model $A$, an index set $I$, and an ultrafilter $\Uc$ on $I$. Fix an $n$-tuple $\abar \in A^\Uc$ corresponding to a family $(\abar^i)_{i \in I}$ and fix some finitely generated extension $\langle \abar \bbar \rangle \supseteq \langle \abar \rangle$. Let $\bbar$ be an $m$-tuple. Fix a $K$ that is large for $\abar\bbar$. By passing to a smaller set of indices if necessary, we may assume that for any $j,k < n$, if $d(a_j,a_k) < \infty$, then $d(a_j^i,a_k^i) < K$ and if $d(a_j,a_k) = \infty$, then $d(a_j^i,a_k^i) > 2K$ for every $i \in I$.

  Fix $i \in I$. There are three cases.
  \begin{itemize}
  \item \emph{Case 1.} If $\langle \abar^i \rangle$ is isomorphic to $\langle \abar \rangle$, then since $A$ is t.e.c., we can find a $\bbar^i \in A$ such that $(\langle \abar^i\bbar^i \rangle,\abar^i\bbar^i)$ is isomorphic to $(\langle \abar\bbar \rangle,\abar\bbar)$. In this case let $L^i = 0$.
  \item \emph{Case 2.} If $\langle \abar^i \rangle$ is not isomorphic to $\langle \abar \rangle$ but $\rKRFR(\langle \abar^i \rangle_K,\abar^i;\langle \abar \rangle,\abar) = 0$, then it must be the case that there are $j,k<n$ such that $d(a_j,a_k) = \infty$ but $d(a^i_j,a^i_k) < \infty$ (by the discussion before the statement of the proposition). Let $L^i$ be $\min\{d(a^i_j,a^i_k):d(a_j,a_k) = \infty\}$. (Note that $L^i > 2K$.) By applying \cref{lem:RFB-extension} inductively $m$ times, we can find a model $\langle \abar^i\bbar^i \rangle \models \RFR$ extending $\langle \abar^i \rangle$ such that $\rKRFR(\langle \abar^i\bbar^i \rangle_K,\abar^i\bbar^i;\langle \abar\bbar \rangle,\abar\bbar) \leq \frac{1}{L^i}$ and for any $j < m$, if $d(b_j,\langle \abar \rangle) = \infty$, then $d(b^i_j,\langle \abar^i \rangle) = \infty$.
  \item \emph{Case 3.} If $\rKRFR(\langle \abar^i \rangle_K,\abar^i;\langle \abar \rangle,\abar) > 0$, then by applying \cref{lem:RFB-extension} inductively $m$ times, we have that for each $i \in I$, there is a model $\langle \abar^i\bbar^i \rangle \models \RFR$ extending $\langle \abar^i \rangle$ such that
    \(
      \rKRFR(\langle \abar^i\bbar^i \rangle_K,\abar^i\bbar^i;\langle \abar\bbar \rangle,\abar\bbar) \leq 5^m \rKRFR(\langle \abar^i \rangle_K,\abar^i;\langle \abar \rangle,\abar)
    \)
    and for any $j < m$, if $d(b_j,\langle \abar \rangle) = \infty$, then $d(b^i_j,\langle \abar^i \rangle) = \infty$.
  \end{itemize}
  Since $A$ is t.e.c., we may furthermore assume that each $\langle \abar^i\bbar^i \rangle$ is actually a substructure of $A$ extending $\langle \abar^i \rangle$. Let $\bbar^\Uc$ be the elements of $A^\Uc$ corresponding to the family $(\bbar^i)_{i \in I}$. One of the 3 cases must happen on a $\Uc$-large set of indices. If case 1 happens on a $\Uc$-large set, then we clearly have that $(\langle \abar\bbar^\Uc \rangle,\abar\bbar^\Uc)$ is isomorphic to $(\langle \abar\bbar \rangle,\abar\bbar)$. If case 2 happens on a $\Uc$-large set, then we must have that $\lim_{i \to \Uc}L^i = \infty$ and therefore $\lim_{i \to \Uc}\rKRFR(\langle \abar^i\bbar^i \rangle_K,\abar^i\bbar^i;\langle \abar\bbar \rangle,\abar\bbar) = 0$. Likewise, if case 3 happens on a $\Uc$-large set, then $\lim_{i \to \Uc}\rho_K(\langle \abar^i \rangle_K,\abar^i;\langle \abar \rangle,\abar) = 0$ by \cref{lem:rho-limit} and therefore $\lim_{i \to \Uc}\rKRFR(\langle \abar^i\bbar^i \rangle_K,\abar^i\bbar^i;\langle \abar\bbar \rangle,\abar\bbar) = 0$ as well.

  By the discussion just before the statement of the proposition, we have that in these second two cases, the pointed $\Lc_{\RFR}$-structures $(\langle \abar\bbar^\Uc \rangle,\mxmn{d}{0}{K},R,\abar\bbar^\Uc)$ and $(\langle \abar\bbar \rangle,\mxmn{d}{0}{K},R,\abar\bbar)$ are isomorphic. Furthermore, we know that for any $b_j \in \bbar$, if $d(b_j,\langle \abar b_{<j} \rangle) = \infty$, then we ensured that $d(b^i_j,\langle \abar b^i_{<j} \rangle) =\infty$. This implies that $(\langle \abar\bbar^\Uc \rangle,\abar\bbar^\Uc)$ and $(\langle \abar\bbar \rangle,\abar\bbar)$ are actually isomorphic as pointed $\Lc_{\RFR}$-structures. Since $\langle \abar \rangle \subseteq \dcl(\abar)$, we have that the isomorphism witnessing this fixes $\langle \abar \rangle$ pointwise.

 So, since we can do this for any $\langle \abar \rangle \subseteq A^\Uc$ and any finite extension $\langle \abar\bbar \rangle\supseteq \langle \abar \rangle$, we have that $A^\Uc$ is t.e.c.
\end{proof}

\begin{defn}
 Let $\RFR^\ast$ be the common theory of t.e.c.\ models of $\RFR$.
\end{defn}

\begin{cor}\label{cor:RFB-model-companion}
  $\RFR^\ast$ is the model companion of $\RFR$. In any model of this theory, the type of a tuple $\abar$ is uniquely determined by the pointed $\Lc_{\RFR}$-isomorphism type of $(\langle \abar \rangle,\abar)$.
\end{cor}
\begin{proof}
  By \cref{prop:RFB-tec-ult}, there is an $\aleph_0$-saturated t.e.c.\ model $A$ of $\RFR^\ast$. By \cref{prop:RFB-model-companion-tec-basic}, for any finite $\abar \in A$, we have that $\tp(\abar)$ is uniquely determined by the pointed isomorphism type of $(\langle \abar \rangle,\abar)$.

  Fix an arbitrary model $B$ of $\RFR^\ast$. By the continuous Keisler-Shelah theorem \cite[Cor.~C.5]{Goldbring2020-GOLCSP-3}, there is an ultrapower $B^\Uc$ of $B$ that is isomorphic to an ultrapower of $A$. Therefore $B^\Uc$ is t.e.c., by \cref{prop:RFB-tec-ult}. For any finite tuple $\bbar \in B$, we have that $\langle \bbar \rangle$ computed in $B$ is the same as $\langle \bbar \rangle$ computed in $B^\Uc$ (since $\langle \bbar \rangle \subseteq \dcl(\bbar)$). Therefore $\tp(\bbar)$ is again uniquely determined by the pointed isomorphism type of $(\langle \bbar \rangle,\bbar)$. This implies that if $C$ and $D$ are models of $\RFR^\ast$ such that $C \subseteq D$, then $C \preceq D$. Therefore, by \cref{prop:RFB-model-companion-tec-basic} part \ref{mcB-1}, $\RFR^\ast$ is the model companion of $\RFR$.
\end{proof}

\begin{prop}\label{prop:RFB-reduct}
  The reduct of $\RFR^\ast$ to the pure metric language is $\RF^\ast$.
\end{prop}
\begin{proof}
  If $A \models \RFR^\ast$ and $B \supseteq A$ is an $\Rb$-forest, then by \cref{lem:RFB-extension}, for any finite $\abar \in A$ and $\bbar \in B$, $\langle \abar\bbar \rangle$ can be expanded to a model of $\RF$ extending $\langle \abar \rangle$. (Specifically, we can consider the model of $\RF$ whose underlying $\Rb$-forest is $\langle \abar\bbar \rangle$ which has $R(x,y) = 0$ for all $x$ and $y$.\footnote{This can also be proven directly without using \cref{lem:RFB-extension} relatively easily, but we have already proven a strictly harder statement, so there is no point is bothering to write this out.}) By compactness and \cref{prop:R-forest-intersection-AE}, this implies that $B$ can be expanded to a model of $\RFR$.

  This implies that for any model $A_0\models \RFR^\ast$ we can build two sequence $(A_i)_{i<\omega}$ and $(B_i)_{i<\omega}$ such that $A_0\subseteq B_0 \subseteq A_1\subseteq B_1\subseteq \dots$ and for each $i<\omega$, $A_i \models \RFR^\ast$ and $B_i \models \RF^\ast$. We then have that the completion of the union is a model of $\RFR^\ast$ whose pure metric reduct is a model of $\RF^\ast$, whence the theory $\RF^\ast$ is the pure metric reduct of the theory $\RFR^\ast$.
\end{proof}

\subsection{Simplicity itself}
\label{sec:simple}

\begin{prop}\label{prop:RFB-unstable}
  $\RFR^\ast$ is unstable.
\end{prop}
\begin{proof}
  The $\Lc_{\RFR}$-structure whose universe is $\{a_i,b_i : i < \omega\}$ with all non-zero distances $\infty$, $R(a_i,a_j) = R(b_i,b_j) = 0$ for all $i,j<\omega$, $R(a_i,b_j) = 0$ if $i < j$, and $R(a_i,b_j) = 1$ if $i \geq j$ is a model of $\RFR$. It can be embedded in a model of $\RFR^\ast$, which is therefore unstable.
\end{proof}

Simplicity in the context of compact abstract theories was defined in \cite{SimpInCATs}. This of course applies to the special case of continuous logic. In particular, we have the following.

\begin{fact}[{\cite[Thm.~1.51]{SimpInCATs}}]\label{fact:CAT-simple}
  A theory $T$ is simple if and only if there is a ternary relation $\indast$ of small subsets of the monster $\Ob$ satisfying the following:
  \begin{itemize}
  \item (Invariance) For any $\sigma \in \Aut(\Ob)$ and any $A$, $B$, and $C$, $B \indast_A C$ if and only if $\sigma B \indast_{\sigma A} \sigma C$. 
  \item (Existence) $B \indast_M M$ for all models $M$.
  \item (Finite character) $B \indast_A C$ if and only if $B_0 \indast_A C_0$ for all finite $B_0 \subseteq B$ and $C_0 \subseteq C$.
  \item (Symmetry) $B \indast_A C$ if and only if $C \indast_A B$.
  \item (Restricted transitivity) For models $M \preceq N$, $B \ind_M NC$ if and only if $B \ind_M N$ and $B \ind_N C$.  
  \item (Local character) There exists a $\lambda$ such that whenever $\cf(\mu) > \lambda+|B|$ and $(M_i)_{i<\mu}$ is an increasing sequence of models, there is some $j<\mu$ such that $B \indast_{M_j}\bigcup_{i<\mu}M_i$. 
  \item (Extension) If $B \indast_A C$ and $C' \supseteq C$, then there is $B' \equiv_A B$ such that $B' \indast_AC'$.
  \item (The independence theorem over models) For any model $M$, if $B_0 \indast_M B_1$, $C_0\indast_M B_0$, $C_1 \indast_M B_1$, and $C_0 \equiv_{M} C_1$, then there is a $C \indast_M B_0B_1$ with $B \equiv_{MA_i}B_i$ for both $i<2$.
  \end{itemize}
  Furthermore, in this case $\indast$ coincides with non-dividing.
\end{fact}

\begin{lem}\label{lem:Lip-check}
  If $A$ is an $\Rb$-forest, $B,C \subseteq A$ are $\Rb$-forests such that $A = B \cup C$, and $f : A \to [0,1]$ is a function that is $1$-Lipschitz on $B$ and $1$-Lipschitz in $C$, then $f$ is $1$-Lipschitz on all of $A$.
\end{lem}
\begin{proof}
  Fix $b \in B$ and $c \in C$. If $d(b,B\cap C) = \infty$ or $d(c,B\cap C) = \infty$, then trivially $|f(b)-f(c)| \leq d(b,c) = \infty$, so assume that $d(b,B\cap C) < \infty$ and $d(c,B\cap C) < \infty$. Let $D = B\cap C$. We now have that
  \begin{align*}
    |f(b)-f(c)| &\leq |f(b)-f(\pi_{D}(b))| + |f(\pi_{D}(b)) - f(\pi_{D}(c))| + |f(\pi_{D}(c))-f(c)| \\
                & \leq d(b,\pi_{D}(b)) + d(\pi_{D}(b),\pi_{D}(c)) + d(\pi_{D}(c),c) = d(b,c).
  \end{align*}
  Therefore $f$ is $1$-Lipschitz on all of $A$.
\end{proof}

\begin{prop}\label{prop:RFB-simple}
  Let $B \indast_A C$ be the ternary relation $\langle AB \rangle \cap \langle AC \rangle = \langle A \rangle$. $\indast$ is a simple independence relation in $\RFR^\ast$ and so $\RFR^\ast$ is simple and $B \ind_A C$ if and only if $\langle AB \rangle\cap \langle AC \rangle = \langle A \rangle$.
\end{prop}
\begin{proof}
  By \cref{prop:RFB-reduct}, the pure metric reduct of $\RFR^\ast$ is $\RF^\ast$. By \cref{prop:RF-stable}, $\RF^\ast$ is stable, and by \cref{prop:RF-ind}, $\indast$ is the relation of non-dividing in $\RF^\ast$. This immediately implies that $\indast$ satisfies all of the properties in \cref{fact:CAT-simple} except for the independence theorem over models.

  So fix a model $M$ and sets $B_0$, $B_1$, $C_0$, and $C_1$ such that $B_0 \indast_M B_1$, $C_i \indast_M B_i$ for both $i<2$, and $C_0 \equiv_M C_1$. Let $C$ be an $\Rb$-forest extending $\langle MB_0B_1 \rangle$ in the same way that $C_0$ and $C_1$ do. (This exists by stationarity of non-dividing in the stable reduct.) There is clearly a well-defined function $Q: \langle MB_0B_1C \rangle^2 \to [0,1]$ defined in such a way that $(\langle MB_iC \rangle,Q,MB_iC)$ is isomorphic as a pointed $\Lc_{\RFR}$\nobreakdash-\hspace{0pt}structure to $(\langle MB_iC_i \rangle,R)$ for both $i<2$. Therefore the only thing we need to check is that $Q$ is \oo-Lipschitz. By \cref{lem:Lip-check}, the functions $x\mapsto Q(x,b)$ and $x\mapsto Q(b,x)$ are $1$-Lipschitz for any $b \in \langle MB_0B_1C \rangle$ (since they are $1$-Lipschitz on $\langle MB_0C \rangle$ and $\langle MB_1C \rangle$). Therefore $Q$ is \oo-Lipschitz and we can assume that $Q$ is $R$ for some model of $\RFR$ extending $\langle MB_0B_1 \rangle$.

  Therefore $\indast$ is a simple independence relation and $\RFR^\ast$ is simple.
\end{proof}

\subsection{\texorpdfstring{$\RFR^\ast$ has no non-algebraic crisp imaginary types}{RFR* has no non-algebraic crisp imaginary types}}
\label{sec:no-crispy}

We will now use \cref{cor:connected-test} to show that $\RFR^\ast$ has no non-algebraic crisp imaginary types. For any type $p(x) \in S_1(A)$, let $F_p$ be $\{q(x,y) \in S_2(A): q(x,y) \supseteq p(x)\cup p(y)\}$.

\begin{lem}\label{lem:unzip}
  For any $A$, any non-algebraic $1$-type $p(x) \in S_1(A)$, and any $q(x,y) \in F_p$, there is a type $r(x,y) \in F_p$ and a continuous path $f:[0,1] \to F_p$ such that $f(0) = q$ and $f(1)=r$ and for any $bc \models r$, $b \ind_A c$.
\end{lem}
\begin{proof}
  Since $\langle A \rangle\subseteq \dcl(A)$, we may assume that $A = \langle A \rangle$. By \cref{prop:RF-JEP-AP} and \cref{cor:RFB-model-companion}, we know that given $b$, $\tp(b/A)$ is uniquely determined by the pointed isomorphism type of $(\langle Ab \rangle,Ab)$. Fix $b$ and $c$ not in $A$ such that $b \equiv_A c$. Since $b \equiv_A c$, we either have that $d(A,b) = d(A,c) = \infty$ or $d(A,b)=d(A,c) < \infty$ and $\pi_A(b) = \pi_A(c)$. In the first case, however, we have that $\langle Ab \rangle\cap \langle Ac \rangle = \langle A \rangle$, so $b \ind_A c$ and we can take $f$ to be a constant path. In the second case, if $\langle Ab \rangle\cap \langle Ac \rangle = \langle A \rangle$, then again $b \ind_A c$ and so we can again take $f$ to be a constant path. Otherwise, there is a unique $e \in \langle Ab \rangle\cap \langle Ac \rangle$  such that $\langle Ab \rangle\cap \langle Ac \rangle = \langle Ae \rangle$. Let $a = \pi_A(b) = \pi_A(c)$.

  For each $r \in [0,d(a,e))$, build $b_r$ and $c_r$ realizing $\tp(b/A)=\tp(c/A)$ and $e_r$ satisfying that
  $d(a,e_r) = r$, 
  $\langle Ab_r \rangle\cap \langle Ac_r \rangle = \langle A e_r \rangle$, and 
  for any $s \in [0,1]$ and $t \in [0,1]$, $R(\interp{a}{s}{b_r},\interp{a}{t}{c_r}) = R(\interp{a}{s}{b},\interp{a}{t}{c})$. 
  Note that the above completely specifies a completion of $\tp(b/A)\cup\tp(c/A)$. It is also straightforward to check that the described behavior of $R$ is \oo-Lipschitz. Furthermore, we have that $b_0 \ind_A c_0$, since $\langle Ab_0 \rangle\cap \langle Ac_0 \rangle = \langle Ae_0 \rangle = \langle A \rangle$.

  So now we just need to verify that the function $f(r) = \tp(b_rc_r/A)$ is continuous (where $b_{d(a,e)}c_{d(a,e)}=bc$). If $(r_i)_{i< \omega}$ is a sequence of elements of $[0,1]$ and $\Uc$ is an ultrafilter on $\omega$, then it is fairly immediate that in an ultrapower of any model containing $A$ and $\{b_r,c_r:r\in [0,1]\}$, a $2$-tuple $b'c'$ corresponding to the family $(b_{r_i}c_{r_i})_{i < \omega}$ will have the property that $(\langle Ab'c' \rangle,b'c')$ is isomorphic to $(\langle Ab_sc_s \rangle,Ab_sc_s)$ as a pointed $\Lc_{\RFR}$-structure, where $s = \lim_{i \to \Uc} r_i$. By \cref{cor:RFB-model-companion}, this implies that $b'c'\equiv_A b_sc_s$. Since we can do this for any $s \in [0,1]$, we have that $f$ is continuous.
\end{proof}

For the following lemma, recall that any convex combination of $1$-Lipschitz functions is $1$-Lipschitz. This implies that any convex combination of \oo-Lipschitz functions is \oo-Lipschitz as well.

\begin{lem}\label{lem:interpolate}
  Fix $p(x) \in S_x(A)$. Fix $q_0(x,y),q_1(x,y) \in F_p$ such that for both $i<2$, if $bc \models q_i$, then $b \ind_A c$. There is a continuous function $g:[0,1] \to F_p$ such that $g(0) = q_0$ and $g(1) = q_1$.
\end{lem}
\begin{proof}
  Find $b_0c_0 \models q_0$ and $b_1c_1 \models q_1$. If $d(A,b_0) = \infty$, then for each $r \in (0,1)$, find $b_r$ and $c_r$ both realizing $p$ such that $R(b_r,c_r) = (1-r)R(b_0,c_0)+rR(b_1,c_1)$ and $R(c_r,b_r) = (1-r)R(c_0,b_0)+rR(c_1,b_1)$. If $d(A,b_0) < \infty$, then let $a = \pi_{\langle A \rangle}(b_0) = \pi_{\langle A \rangle}(c_0) = \pi_{\langle A \rangle}(b_1) = \pi_{\langle A \rangle}(c_1)$. For each $r \in (0,1)$, find $b_r$ and $c_r$ both realizing $p$ such that for any $s,t \in [0,1]$,
  \[
    R(\interp{a}{s}{b_r},\interp{a}{t}{c_r}) = (1-r)R(\interp{a}{s}{b_0},\interp{a}{t}{c_0})+rR(\interp{a}{s}{b_1},\interp{a}{t}{c_0}). 
  \]
  Since convex combinations of \oo-Lipschitz functions are \oo-Lipschitz, this specifies a complete type in $F_p$ for each $r \in (0,1)$.

  In both cases we can use the same argument as in \cref{lem:unzip} to show that the function $g(r) = \tp(b_rc_r/A)$ is continuous.
\end{proof}

We don't need this but one thing to note about Lemmas~\ref{lem:unzip} and \ref{lem:interpolate} is that they do not actually require that $b$ and $c$ realize the same type over $A$.

\begin{thm}\label{thm:no-crispy}
  $\RFR^\ast$ is a simple, unstable continuous theory with no non-algebraic crisp imaginary types.
\end{thm}
\begin{proof}
  $\RFR^\ast$ is unstable by \cref{prop:RFB-unstable} and simple by \cref{prop:RFB-simple}. By Lemmas~\ref{lem:unzip} and \ref{thm:no-crispy}, $F_p$ is path connected for any $1$-type $p(x) \in S_1(A)$ for any set of parameters $A$. Therefore by \cref{cor:connected-test}, $\RFR^\ast$ has no non-algebraic crisp imaginary types.
\end{proof}

Given the discussion in the introduction of \cite{Ealy2012}, it is natural to wonder whether adding a generic automorphism to $\RF^\ast$ results in a strictly simple theory.

\begin{quest}
  Does the theory of models of $\RF^\ast$ with an automorphism have a model companion? If it does, is this theory strictly simple?
\end{quest}

\section{A stronger notion}
\label{sec:quests}

There is a sense in which models of $\RF^\ast$ are `less essentially continuous' than, say, Hilbert spaces or the Urysohn sphere. In particular, given any point $a$ in a model of $\RF^\ast$ and $r \in (0,\infty)$, the set $\{x : d(x,a) = r\}$ is definable and is an ultrametric space with a dense distance set. While as we have shown, the induced structure on this ultrametric space does not interpret an infinite discrete structure, there is still an abundance of type-definable and co-type-definable equivalence relations whose equivalence classes are uniformly metrically separated (specifically, $d(x,y) \leq r$ and $d(x,y)<r$ for any $r$).  Similarly, there is also the discrete equivalence relation given by $d(x,y) < \infty$ in any $\Rb$-forest. In some sense, rather than having too few uniformly metrically discrete equivalence relations to interpret an infinite discrete structure, $\RFR^\ast$ has too many. By contrast, non-compact ultrametric structures with nowhere dense distance sets\footnote{The \emph{distance set of $T$} is the set of distances $d(a,b)$ of pairs of elements of models of $T$.} interpret infinite discrete structures (see \cite[Sec.~5.2.3]{CatCon}).

A salient feature differentiating the $\Rb$-arboreal theories (i.e., $\Rb$-trees and $\Rb$-forests) from Hilbert spaces, atomless probability algebras, and the Urysohn sphere is connectivity. In particular, the theories of infinite dimensional Hilbert spaces and atomless probability algebras enjoy the following property:
\begin{itemize}
\item[$(\star)$] For any small partial type $\Sigma(\xbar)$ (in finitely many variables), the metric space of realizations of $\Sigma(\xbar)$ in the monster has a bounded number of connected components.
\end{itemize}
One thing to note is that $(\star)$ rules out any theory (with non-compact models) whose models are all locally compact, including the `inherently non-discrete' strongly minimal theories of \cite[Sec.~5.1]{CatCon}.

Regardless, $(\star)$ seems reasonably natural but also harder to characterize than the property of having no non-algebraic crisp imaginary types. We can say a little about it, though.

\begin{prop}\label{prop:dag-equiv}
  Let $T$ be a theory with monster model $\Ob$. The following are equivalent.
    \begin{enumerate}
    \item $T$ has $(\star)$.
    \item For every model $M$ and every complete type $p(\xbar) \in S(M)$, $p(\Ob)$ has a bounded number of connected components.
    \item For every model $M$ and every complete type $p(\xbar) \in S(M)$, $p(\Ob)$ has a single connected component.
    \end{enumerate}
  Furthermore, if $T$ has $(\star)$, then it has no non-algebraic crisp imaginary types.
\end{prop}
\begin{proof}
  1 $\To$ 2 is obvious.

  $\neg$3 $\To$ $\neg$2. Suppose that for some model $M$ and some complete type $p(\xbar) \in S(M)$, there are realizations $a$ and $b$ of $p$ such that $a$ and $b$ are in different connected components of $p(\Ob)$. Let $c$ realize some $M$-coheir $q(\xbar)$ of $p(\xbar)$ over $Mab$. Since being in the same connected component of $p(\Ob)$ is an invariant equivalence relation, we must have that either $c$ is not in the same connected component as $a$ or that $c$ is not in the connected component as $b$. By invariance, this implies that $c$ is not in the connected components of either $a$ or $b$. Now any Morley sequence $(c_i)_{i<\kappa}$ in $q(\xbar)$ will have the property that for any $i<j<\kappa$, $c_i$ and $c_j$ are not in the same connected component of $p(\Ob)$, implying that $p(\Ob)$ has an unbounded number of connected components.

  $\neg$1 $\To$ $\neg$3. Suppose that there is a small partial type $\Sigma(\xbar)$ such that $\Sigma(\Ob)$ has a large number of connected components. Fix a model $M$ containing all of the parameters of $\Sigma(\xbar)$. Let $(a_i)_{i<(2^{|M|+|T|})^+}$ be a collection of elements of $\Ob$ in pairwise distinct connected components of $\Sigma(\Ob)$. By the pigeonhole principle, there is a completion $p(\xbar) \supseteq \Sigma(\xbar)$ realized by $a_i$ and $a_j$ for $i<j < (2^{|M|+|T|})^+$.

  For the final statement, note that if $T$ has a non-algebraic crisp imaginary type, then by \cref{thm:crisp-imaginary-char-1-types}, there is a complete type $p(x)$ and a $\{0,1\}$-valued definable equivalence relation $E(x,y)$ on $p(\Ob)$ with a large number of $E$-classes. Any $a,b \models p$ with $\neg E(a,b)$ must be in distinct connected components of $p(\Ob)$, so $T$ has $(\star)$.
\end{proof}

One thing to note is that the first part of \cref{prop:dag-equiv} doesn't use anything special about the relation `$a$ and $b$ are in the same connected component of $p(\Ob)$' other than the fact that it is an invariant equivalence relation on $p(\Ob)$. By comparison to \cref{thm:crisp-imaginary-char-1-types}, a reasonable test question for the robustness of $(\star)$ would be this:

\begin{quest}
  Is the failure of $(\star)$ always witnessed by $1$-types?
\end{quest}

And given the results of \cref{sec:gen-bin-R-tree} and the fact that $(\star)$ separates Hilbert spaces, atomless probability algebras, and the Urysohn sphere\footnote{$L^p$-Banach lattices probably have $(\star)$ too, but we have not checked this.} from the $\Rb$-arboreal theories, the next question is natural.

\begin{quest}\label{quest:star-simple}
  Is there a strictly simple theory with $(\star)$?
\end{quest}

Characterizing $(\star)$ more precisely might require an analysis of the topological properties of sets of types of the form $\{q(\xbar,\ybar) \in S_{\xbar\ybar}(M) : q(\xbar,\ybar) \supseteq p(\xbar)\cup p(\ybar),~(\forall ab\models q)a~\text{and}~b~\text{are in the same connected component of}~p(\Ob)\}$, which seems difficult. By \cite{Debs2020}, the complexity of the class of connected closed subset of a Polish space can be as high as $\boldsymbol{\Pi}^1_2$, but even worse it was established in \cite{con-not-sep-con} that connectivity between points in a complete metric space is not always witnessed by separable subspaces. It's not even clear that $(\star)$ is set-theoretically absolute. Perhaps this will be the kind of ugliness that is ironed out by considering the class of all models of a theory, but perhaps not.

\begin{quest}
  Is $(\star)$ set-theoretically absolute? What is the topological complexity of the set of complete theories satisfying $(\star)$?
\end{quest}

One can write down a number of reasonable conditions intermediate between $(\star)$ and having no non-algebraic crisp imaginary types---for instance, we could consider path-connected components or quasi-components instead of connected components or we could think about $\e$-chainability as a uniform proxy for connectivity---but given that there are no strongly motivating examples separating these intermediate notions, it seems premature to commit them to print.

\bibliographystyle{plain}
\bibliography{../ref}

\end{document}